\documentclass{article}
\usepackage[paper=a4paper,left=30mm,right=30mm,top=25mm,bottom=25mm]{geometry}
\usepackage{amsmath,amsfonts,amssymb,amsthm,booktabs,epsfig,hyperref}
\usepackage{graphicx}
\usepackage{epsfig}
\usepackage[numbers]{natbib}
\usepackage{color}
\usepackage{dsfont}
\usepackage{bm}
\usepackage{verbatim}
\usepackage{stmaryrd}
\usepackage{bbm}
\usepackage{enumerate}
\usepackage{enumitem}
\usepackage{nicefrac}
\usepackage{xfrac}
\usepackage{units}
\usepackage{mathtools}
\usepackage{pdfcomment}
\usepackage{url}
\usepackage{acronym}
\usepackage{mathrsfs}
\usepackage{fancyhdr}
\usepackage[T1]{fontenc}
\usepackage{caption, booktabs}
\usepackage{textcomp}
\usepackage{multirow}
\usepackage{xcolor}

\theoremstyle{plain}
\newtheorem{theorem}{Theorem}

\newtheorem{proposition}{Proposition}
\newtheorem{lemma}{Lemma}
\newtheorem{corollary}{Corollary}

\theoremstyle{definition}

\newtheorem{remark}{Remark}
\newtheorem{example}{Example}
\newtheorem{optimisationproblem}{Optimisation Problem}

\def\1{\mathds{1}} 
\makeatletter
\renewcommand\d[1]{\ensuremath{\;\mathrm{d}#1\@ifnextchar\d{\!}{}}} 
\newcommand\nopagebreakhere{\par\nobreak\@afterheading}  
\makeatother

\newcommand{\de}{\mathrm{\,d}}

\newcommand{\R}{\mathbb{R}}

\newcommand{\ind}{\mathds{1}}

\newcommand{\Var}{\mathrm{Var}}

\newcommand{\supp}{\mathop{\mathrm{supp}}}

\title{The exact region and an inequality between Chatterjee's and Spearman's rank correlations
}
\author{%
  Jonathan Ansari\textsuperscript{*}
  \and
  Marcus Rockel\textsuperscript{$\dagger$}
}

\begin{document}

\maketitle
\begin{center}
  \small\textit{%
    \textsuperscript{*}Department of Mathematics,\\
    University of Salzburg,\\
    Hellbrunner Straße 34, 5020 Salzburg, Austria,\\
    \texttt{jonathan.ansari@plus.ac.at}
  }\\[2mm]
  \small\textit{%
    \textsuperscript{$\dagger$}Department of Quantitative Finance,\\
    Institute for Economics, University of Freiburg,\\
    Rempartstr. 16, 79098 Freiburg, Germany,\\
    \texttt{marcus.rockel@finance.uni-freiburg.de}
  }
\end{center}
\begin{abstract}
The rank correlation \(\xi(X,Y),\) recently established by Sourav Chatterjee and already popular in the statistics literature, takes values in \([0,1],\) where \(0\) characterises independence of \(X\) and \(Y,\) and \(1\) characterises perfect dependence of \(Y\) on \(X.\)
Unlike concordance measures such as Spearman’s \(\rho,\) which capture the degree of positive or negative dependence, \(\xi\) quantifies the strength of functional dependence.
In this paper, we study the attainable set of pairs \((\xi(X,Y),\rho(X,Y))\).
The resulting \(\xi\)-\(\rho\)-region is a convex set whose boundary is characterised by a novel family of absolutely continuous, asymmetric copulas having a diagonal band structure.
Moreover, we prove that \(\xi(X,Y)\le |\rho(X,Y)|\) whenever \(Y\) is stochastically increasing or decreasing in \(X,\) and we identify the maximal difference \(\rho(X,Y)-\xi(X,Y)\) as exactly \(0.4.\) Our proofs rely on a convex optimisation problem under various equality and inequality constraints, as well as on ordering properties for \(\xi\) and \(\rho.\)
Our results contribute to a better understanding of Chatterjee’s rank correlation, which typically yields substantially smaller values than Spearman's rho when quantifying positive dependencies. In particular, when interpreting the values of Chatterjee’s rank correlation on the scale of \(\rho\), the quantity \(\sqrt{\xi}\) appears to be more appropriate.
\end{abstract}

\vspace{0ex}

\section{Introduction and main results}\label{sec1}

The most famous rank correlations are certainly Spearman's rho and Kendall's tau. They attain values in the interval \([-1,1],\) where the values \(+1\) and \(-1\) characterise perfect positive and negative dependence, respectively. 
Given a bivariate random vector \((X,Y)\) with continuous marginal distribution functions \(F_X\) and \(F_Y,\) Spearman's rho is defined as the Pearson correlation of \((F_X\circ X,F_Y \circ Y).\) 
It admits a representation in terms of the copula \(C\) of \((X,Y)\) through
\begin{align}\label{frm_rho_integral}
    \rho(X,Y)=\rho(C) = 12 \int_{[0,1]^2} C(u,v) \de \lambda^2(u,v) - 3;
\end{align}
see, e.g., \cite[Theorem~5.1.6]{Nelsen-2006}.
Both Spearman's rho and Kendall's tau are measures of concordance and thus quantify the strength of positive or negative dependence between \(X\) and \(Y.\) 

Recently, \citet{chatterjee2020} and \citet{chatterjee2021} introduced a new rank correlation that has attracted a lot of attention in the statistics literature; see, e.g., \cite{auddy2021,bickel2022,bickel2020,bernoulli2021,fan2022A,deb2020b,han2022limit,shi2021,Strothmann-2022,wiesel2022}.  
In the bivariate case, the population version of Chatterjee's xi is defined by
\begin{align}\label{eqchatt}
    \xi(X,Y) := \frac{\int_{\mathbb{R}} \Var(P(Y\geq y\mid X)) \de P^Y(y) }{\int_{\mathbb{R}} \Var(\1_{\{Y\geq y\}}) \de P^Y(y)};
\end{align}
see also \cite{Gamboa-Klein-Lagnoux-2018}.
The quantity \(\xi\) takes values in the interval \([0,1],\) where the value \(0\) characterises independence of \(X\) and \(Y\), while the value \(1\) is attained if and only if \(Y\) perfectly depends on \(X,\) i.e., there exists a measurable function \(f\) (that is not necessarily increasing or decreasing) such that \(Y = f(X)\) almost surely.
The motivation for considering \(\xi\) is to quantify the strength of (functional) dependence of \(Y\) on \(X\), which is useful in many respects.
First, in contrast to measures of concordance such as \(\rho\), Chatterjee’s xi is able to detect non-monotone relationships, which is of fundamental importance for statistical modeling.
Moreover, multivariate generalizations of \(\xi\) satisfy additional remarkable properties, such as information monotonicity and characterisation of conditional independence.
These properties enable important applications, including model-free feature ranking, forward feature selection, and clustering methods; see, e.g., \cite{chatterjee2021,Ansari-Fuchs-2022,Fuchs-Wang-2024,deb2020b}.
Noting that \(\xi\) has a fast (nearest-neighbor-based) estimator, it is arguably the simplest dependence measure with this combination of properties; see \cite{Chatterjee-2024} for an overview.

In the case of continuous marginal distribution functions, Chatterjee's xi depends only on the underlying copula \(C\), and it coincides with the Dette-Siburg-Stoimenov measure in \cite{dette2013copula}, i.e.,
\begin{align}\label{repxicop}
   \xi(X,Y) = \xi(C) := 6 \int_0^1 \int_0^1 (\partial_1 C(t,v))^2 \de t \de v - 2.
\end{align}
The behaviour of \(\xi\) is fundamentally different from that of Spearman's rho; see Fig.~\ref{fig:xi_and_rho}. If \(X\) and \(Y\) are independent, then both measures are zero. 
Similarly, if \(Y = \pm X,\) then \(\rho(X,Y) = \pm 1\) and \(\xi(X,Y) = 1.\)
However, it is possible that \(\xi\) attains the maximal value \(1\) in the same situation where \(\rho\) is \(0\); for example, when \(X\) is standard normal and \(Y = |X|.\)
While the properties of Spearman's rho are well-studied in the literature, there are still several open questions concerning the behaviour and interpretation of \(\xi.\)
Noting that \(\xi\) and \(\rho\) quantify fundamentally different aspects of dependence, natural questions concern how they are related and how far they can differ.
More precisely, given the value of \(\xi(X,Y)\), it is important to know which values \(\rho(X,Y)\) can attain.
Vice versa, one may ask how weak or strong the dependence between \(Y\) and \(X\) can be in terms of \(\xi\) when the value of \(\rho\) is fixed.
As we illustrate in Fig.~\ref{fig:rho_xi_comparison}, a small value of \(|\rho|\) can be misleading, since a strong functional relationship may still be present.
By contrast, a large value of \(|\rho|\) indicates a strong monotone association, also reflected in a large value of \(\xi\).
In other words, \(\xi(X,Y)\) quantifies the strength of functional dependence of \(Y\) on \(X\) \cite[Theorem~1.1]{chatterjee2020}, whereas \(\rho\) is sensitive only to monotone dependencies.

In the following theorem, which is our first main result, we answer these questions by determining the exact \(\xi\)-\(\rho\)-region. As illustrated in Fig.~\ref{fig:attainable_rho_xi_region}, for \(\xi = 1,\) the set of possible values for \(\rho\) is the whole interval \([-1,1].\) If \(\xi\) decreases, the attainable interval for \(\rho\) shrinks until \(\rho = 0\) for \(\xi = 0.\) 
The boundary of the \(\xi\)-\(\rho\)-region is uniquely described by a new family of stochastically monotone, asymmetric, and absolutely continuous copulas with a diagonal band structure; see Fig.~\ref{fig:C_x}. We discuss this copula family in more detail in Section~\ref{sec4}. 
We denote by \(\mathcal{C}\) the set of bivariate copulas.

\begin{theorem}[Exact \(\xi\)-\(\rho\)-region]\label{thexirho}
The exact \(\xi\)-\(\rho\)-region is
\begin{align}\label{eqxirhoregion}
    \mathcal{R}:= \bigl\{(\xi(C),~\rho(C)) : C\in\mathcal{C}\bigr\}
  \;=\;
  \bigl\{(x,y)\in\mathbb{R}^2 : x\in[0,1],\; |y|\le M_x\bigr\},
\end{align}
where
\begin{equation}\label{eq:b_and_M}
\begin{aligned} 
M_x = 
\begin{cases}
0,
   & x=0,\\[1.5ex]
b_x-\dfrac{3b_x^2}{10}, 
 \phantom{1\dfrac{1}{2b_x^2}+\dfrac{1}{5b_x^3}}  b_x = \dfrac{\sqrt{6x}}{2\cos\left(\tfrac13\arccos\bigl(-\tfrac{3\sqrt{6x}}{5}\bigr)\right)},
   & 0<x\le\tfrac{3}{10},\\[4ex]
1-\dfrac{1}{2b_x^2}+\dfrac{1}{5b_x^3}, 
\phantom{b_x\dfrac{3b_x^2}{10}} b_x = \dfrac{5+\sqrt{5(6x-1)}}{10(1-x)}, 
  & \tfrac{3}{10}<x<1.\\
 1 & x = 1,
\end{cases} 
\end{aligned}
\end{equation}

The set \(\mathcal{R}\) is convex and its boundary is described by the copula family $(C_b)_{b\in{\mathbb{R}}\setminus \{0\}}$ defined in \eqref{eq:C_x} and \eqref{defcopbneg}.
More precisely, for \(x\in (0,1)\) and for \(b = b_x\) in \eqref{eq:b_and_M},
the copula \(C_{\pm b}\) is the unique copula \(C\in \mathcal{C}\) with \(\xi(C) = x\) and \(\rho(C) = \pm M_x.\)
\end{theorem} 

Elementary calculations show that the function \(x\mapsto M_x-x\) is continuous on \([0,1]\), strictly increasing on \([0,3/10]\), and strictly decreasing on \([3/10,1].\) Hence,
the maximal difference \(\rho(C)-\xi(C)\) over all bivariate copulas \(C\) is attained by the copula \(C_b\) for \(b = b_{0.3} = 1.\) 
Interestingly, the values of \(\rho\) and \(\xi\) for this copula only have one decimal place; see Proposition \ref{propcfe} for closed-form expressions.

\begin{corollary}[Global $\rho-\xi$-maximum]\label{cor:max-rho-minus-xi}
We have \(\max \{\rho(C)-\xi(C) \mid C\in\mathcal{C} \} =  0.4,\)
and the unique maximiser is the copula $C_{1}$ in \eqref{eq:C_x} with \(\rho(C_1) = 0.7\) and \(\xi(C_1) = 0.3.\) 
\end{corollary}

\begin{figure}[t!]
\centering
\includegraphics[width=0.7\textwidth, trim= 0mm 15mm 00mm 25mm, clip]{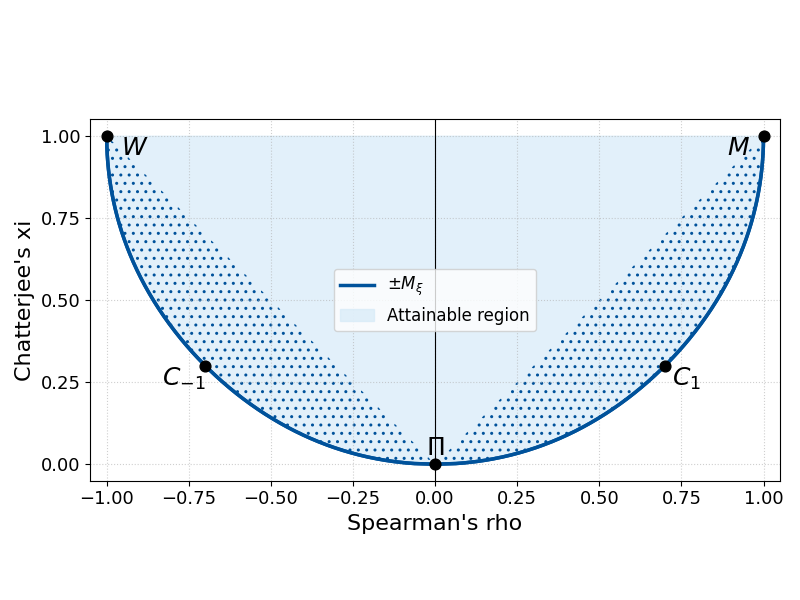}
\caption{
Illustration of the exact \(\xi\)-\(\rho\)-region \(\mathcal{R}\) (transposed) in \eqref{eqxirhoregion}, where stochastically increasing (decreasing) copulas are located in the right (left) scattered area; see Theorem~\ref{thm:rho_ge_xi}. The right and left boundary \((\pm M_\xi,\xi)\) of the region is described by the copula family \((C_b)_{b\in {\mathbb{R}}\setminus\{0\}}\) with limiting cases \(\Pi(u,v) := uv\) for \(b\to 0,\) \(M(u,v):= \min\{u,v\}\) for \(b\to \infty,\) and \(W(u,v) := \max\{u+v - 1,0\}\) for \(b\to -\infty;\) see Theorem~\ref{thexirhooptimisation} and Remark \ref{remsymm}~\ref{remsymma}.
  }
\label{fig:attainable_rho_xi_region}
\end{figure}

Measures of concordance such as Spearman's rho and Kendall's tau are increasing in the pointwise order of copulas; see, e.g., \cite[Remark 6.3]{Rueschendorf-2013}. The uniquely determined maximal (minimal) element in the pointwise order is the upper (lower) Fr\'{e}chet copula \(M\) (\(W\)), which models perfect positive (negative) dependence. Exactly in this case, \(\rho\) attains the value \(+1\) (resp. \(-1\)). 
In contrast, Chatterjee's rank correlation is increasing with respect to the Schur order for copula derivatives recently investigated in \cite{Ansari-Rueschendorf-2021,Ansari-Rockel-2023}. 
The global minimal element of this order is the independence copula \(\Pi\), whereas maximal elements are perfect dependence copulas.
If \(Y\) is stochastically increasing (SI) in \(X\) (i.e., for all \(y\), \(P(Y\geq y \mid X=x)\) is increasing in \(x\)), then the pointwise order and the Schur order for copulas coincide; see \cite[Lemma 2.6 (ii)]{Ansari-Rockel-2023}. Since SI is a positive dependence concept, in this case both \(\xi\) and \(\rho\) quantify the strength of positive dependence. A natural question is then how the values of \(\xi\) and \(\rho\) are related assuming that the underlying copula \(C\) is SI (i.e., for \((X,Y)\sim C\), \(Y\) is SI in \(X\)).

Due to the following theorem---our second main result---Chatterjee's xi is for SI copulas always smaller than Spearman's rho except for the case of independence or perfect positive dependence, where their values coincide. 
By symmetry, a similar statement holds true for stochastically decreasing (SD) copulas (where \(C\) is SD if, for \((X,Y)\sim C,\) \(Y\) is SI in \(-X\)).
For an illustration of this result, we refer to Fig.~\ref{fig:attainable_rho_xi_region}, where the SI and SD copulas form a subset of the scattered area.

\begin{theorem}[SI implies $\xi\le \rho$]\label{thm:rho_ge_xi}
If the copula $C$ is SI or SD, then
\begin{align}\label{eqthm:rho_ge_xi}
    \xi(C)\;\le\;|\rho(C)|,
\end{align}
with equality if and only if $C\in\{W,\,\Pi,\,M\}.$
\end{theorem}

\begin{remark}\label{remmain}
\leavevmode
\begin{enumerate}[label=(\alph*)]
    \item The proof of Theorem~\ref{thexirho} is based on solving a convex maximisation problem over a family of conditional distribution functions; see Optimisation Problem \ref{optprob} and Theorem~\ref{thexirhooptimisation}. Key is the representation of \(\rho(C)-\xi(C)\) as an integral over the partial copula derivative \(h_v(t) := \partial_1 C(t,v);\) see \eqref{optprob1a}. In particular, we exploit the fact that Spearman’s rho can be expressed as a functional of the first partial derivative of the copula, namely
    \begin{align}\label{reprhocopder}
        \rho(C) = 12 \int_0^1 \int_0^1 (1-t) \, \partial_1 C(t,v) \de t \de v - 3;
    \end{align}
    combine \eqref{frm_rho_integral} and \eqref{reprho2}. We refer to \cite{Fredricks-Nelsen-2007} for a related representation of \(\rho\) involving partial derivatives with respect to both components.
    
    For the proof of Theorem~\ref{thm:rho_ge_xi}, we use a maximum principle and show that the difference \(\rho(C)-\xi(C)\) over all SI copulas is minimal for the extreme cases of independence and perfect positive dependence. These cases correspond to the copulas \(C = \Pi\) and \(C= M,\) respectively, for which the difference vanishes.
    \item For any SI copula \(C\), it is well known that \(\rho(C) \geq \tau(C) ,\) where \(\tau(C) = 4 \int C(u,v) \de C(u,v) - 1\) is the copula-based version of Kendall's tau, see, e.g., \cite[Theorem~5.2.8]{Nelsen-2006}. Due to Theorem~\ref{thm:rho_ge_xi}, we have \(\rho(C) \geq \xi(C).\) Simulations indicate that even \(\tau(C)\geq \xi(C)\) holds true; see also Fig. \ref{fig:differences_curves} and \cite{Ansari-Rockel-2023}. However, we were not able to prove this conjecture noting that Kendall's tau admits the representation \(\tau(C)= 1 - 4 \int_0^1\int_0^1 \partial_1 C(u,v) \partial_2 C(u,v) \de u\de v,\) which has partial derivatives of \(C\) with respect to both components.
    \item On the one hand, the maximal difference between Spearman's rho and Chatterjee's xi due to Corollary \ref{cor:max-rho-minus-xi} seems surprisingly large.
    The unique maximiser \(C_1\) is an SI copula, so that both \(\rho\) and \(\xi\) quantify the strength of positive dependence of \(C_1.\)
    In this respect, the value \(\rho(C_1) = 0.7\) suggests a rather strong positive dependence structure between \(X\) and \(Y\) while \(\xi(C_1) = 0.3\) could be interpreted as a rather weak dependence of \(Y\) on \(X.\)
    On the other hand, there are several classical copula families with parameters for which the difference \(\rho-\xi\) is close to the maximal value \(0.4;\) see Table \ref{tab:rho_minus_xi_max} and Fig. \ref{fig:differences_curves} as well as \cite{Ansari-Rockel-2023}.
    From this point of view, one might expect the maximal difference to be larger.
    Considering the representations~\eqref{repxicop} and~\eqref{reprhocopder}, there is a simple explanation for why \(\rho(C)\) attains substantially larger values than \(\xi(C)\) in the case of stochastically monotone copulas:
The copula derivative \(\partial_1 C\) enters the representation of \(\rho(C)\) linearly, whereas it appears squared in the representation of \(\xi(C)\).
From this perspective, considering \(\sqrt{\xi(C)}\) rather than \(\xi(C)\) may be more appropriate, since it leads to values that are closer to those of \(\rho(C)\) for stochastically monotone relationships.
However, in the statistical literature and in practice, almost exclusively the quantity \(\xi\) is considered; see the references cited above and, for instance, the \textsf{R} packages \texttt{FOCI} \cite{FOCI}, \texttt{XICOR} \cite{XICOR}, and \texttt{didec} \cite{didec}.
We therefore focus on \(\xi\) and document the substantial differences between the values of \(\xi\) and \(\rho\) in our work.
    \item Corollary \ref{cor:max-rho-minus-xi} also yields that, if \(\xi(X,Y)\) is larger than \(|\rho(X,Y)|,\) then \(Y\) is not stochastically increasing or decreasing in \(X.\) As a consequence, in this case, \(Y\) does not follow a linear model \(a + bX + \varepsilon\) with \(\varepsilon\) independent of \(X.\)
    \item A comparison of the rank-based estimators $\xi_n$ (for Chatterjee's xi) and $\rho_n$ (for Spearman's rho) is studied by \cite{Zhang-2025}. Extremal differences are obtained for specific permutations of the ranks, where, for a sample size of $n = 100,$ an example is constructed with $\xi_n \approx 0.058$ and $\rho_n \approx 0.753$; see \cite[Extremal case 2]{Zhang-2025}. Similar differences also exist for arbitrarily large finite samples.
    In contrast, since both $\xi_n$ and $\rho_n$ are strongly consistent estimators, we have \(\rho_n-\xi_n \xrightarrow[n\to\infty]{a.s.} \rho-\xi \leq 0.4\) due to Corollary \ref{cor:max-rho-minus-xi}. Hence, Zhang’s construction illustrates an extreme finite-sample phenomenon, which our result complements by showing that the gap cannot be larger than $0.4$ in the limit. 
    \item \label{remmainF} It is worth mentioning that Chatterjee's rank correlation admits a representation in terms of Spearman's footrule via \(\xi(C) = \psi(C^T\ast C)\) (see e.g. \cite{fuchs2021quantifying}), where Spearman's footrule is defined by \(\psi(D) = 6 \int_0^1 D(t,t) \de t-2\). 
    Here, \(D\ast E\,(u,v) := \int_0^1 \partial_2 D(u,t)\, \partial_1 E(t,v)\de t\) denotes the Markov product of two bivariate copulas \(D\) and \(E\), and \(D^T(u,v) := D(v,u)\) is the transposed copula. Hence, properties of Chatterjee’s rank correlation are closely related to properties of Markov products.
    Further, Chatterjee’s xi can be viewed as a dependence measure for two variables based on a proportional reduction in variation.
Measures of this type are closely related to the coefficient of determination \(R^2\) in simple linear regression, where \(R^2\) is given by the square of the correlation coefficient.
This connection is explicitly discussed in \cite{Shih-2021,Shih-Emura-2025}.
In particular, \cite[Section~2.4.3]{Agresti-2013} establishes a link between \(R^2\) and quantities of the form
\((V(Y) - \mathbb{E}[V(Y | X)])/V(Y)\), where \(V(Y)\) denotes the Shannon entropy of \(Y\).
Moreover, \cite[Exercise~2.39]{Agresti-2013} discusses Goodman--Kruskal’s tau for two nominal categorical variables, corresponding to the choice of \(V(Y)\) as a quadratic entropy measure.
\end{enumerate}
\end{remark}

\begin{table}[tb]
  \caption{
  Parameter values that (approximately) maximise the gap $\rho-\xi$ for the listed SI/SD-copula families, together with the corresponding values for Spearman's rho, Chatterjee's xi, and the square root of Chatterjee's xi.
  The entries for the Gaussian, the Marshall-Olkin and the diagonal band copula family $(C_b)_{b\in{\mathbb{R}}\setminus\{0\}}$ are obtained from closed-form expressions for \(\rho\) and \(\xi\); see \cite{fuchs2021quantifying} and \cite[Table 6]{Ansari-Rockel-2023}. The entries for the remaining copula families are computed via a dense grid search over the parameter space, combined with stochastic estimations of $\rho$ and $\xi$. More precisely, for each copula family, we consider 200 different parameter values, and for each parameter value, we generate $2\,000\,000$ data points by leveraging the fast sampling algorithm for completely monotone Archimedean copulas as in \cite[Alg.~7.53]{Embrechts-2015}.
  Unlike the other families considered, the Marshall–Olkin copula has a singular component, which may explain its comparatively higher value of \(\xi\).
  }
  \centering
\begin{tabular*}{\textwidth}{@{\extracolsep{\fill}} lcccccc @{}}
  \toprule
  Family & Param. & $\rho$ & $\xi$ & $\sqrt{\xi}$ & $\rho-\xi$ 
  \\
  $(C_b)_{b\in\mathbb{R}\setminus\{0\}}$
    & \(1\phantom{.000}\)
    & \(0.7\phantom{00}\)
    & \(0.3\phantom{00}\)
    & \(0.548\)
    & \(0.4\phantom{00}\) 
    \\
  Frank
    & \(5.529\)
    & \(0.682\)
    & \(0.299\)
    & \(0.547\)
    & \(0.383\) 
    \\
  Gaussian
    & \(1/\sqrt{2}\)
    & \(0.690\)
    & \(0.310\)
    & \(0.557\)
    & \(0.380\) 
    \\
  Gumbel--Hougaard
    & \(1.991\)
    & \(0.681\)
    & \(0.313\)
    & \(0.559\)
    & \(0.367\) 
    \\
  Clayton
    & \(1.998\)
    & \(0.682\)
    & \(0.335\)
    & \(0.579\)
    & \(0.347\) 
    \\
  Marshall-Olkin (\(\alpha_1=1\))
    & \(0.5\phantom{00}\)
    & \(0.6\phantom{00}\)
    & \(0.4\phantom{00}\)
    & \(0.632\)
    & \(0.2\phantom{00}\) 
    \\
  \bottomrule
\end{tabular*}
  \label{tab:rho_minus_xi_max}
\end{table}

\begin{figure}[t!]
    \centering
    \includegraphics[width=0.99\textwidth, trim={0 15 0 0}, clip]{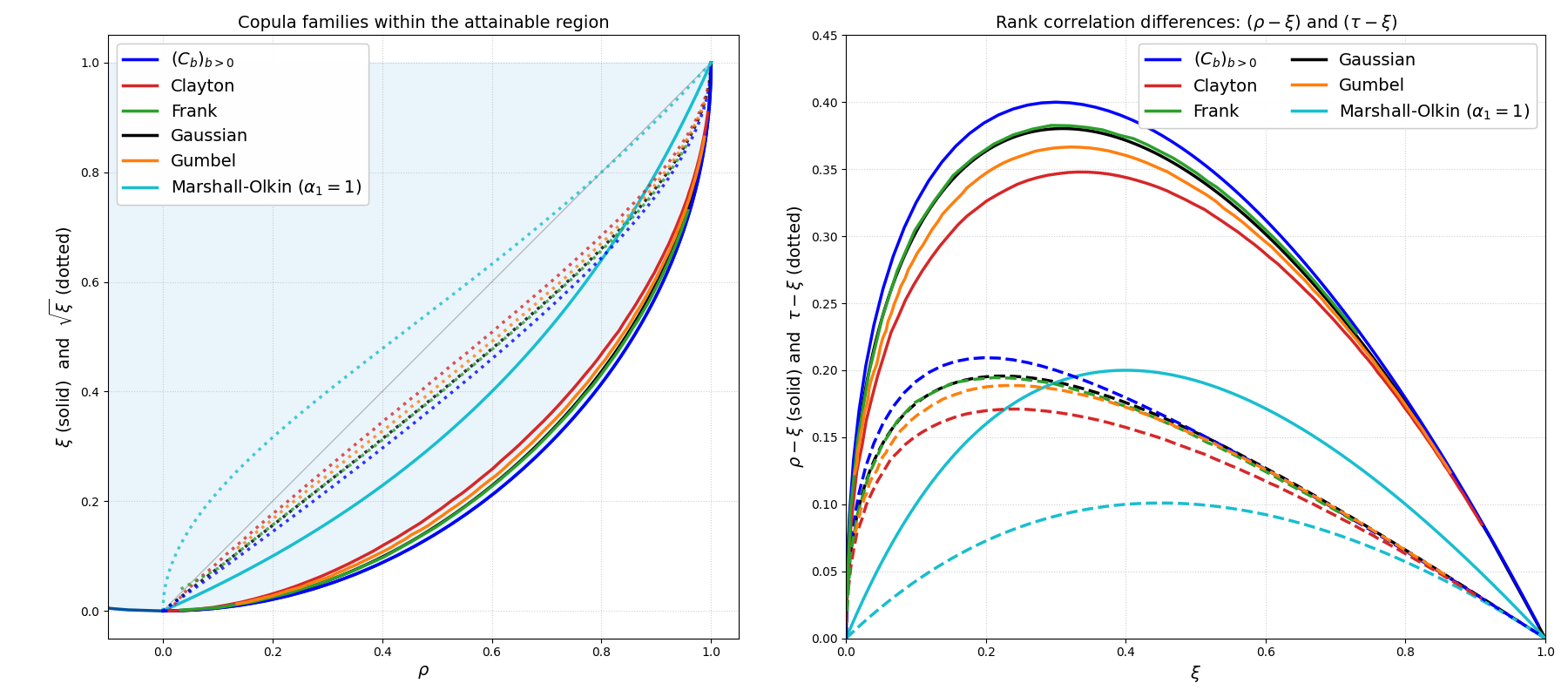}
    \caption{
        Left: $(\rho, \xi)$ (solid) and $(\rho, \sqrt{\xi})$ (dotted) for classical stochastically increasing copula families. Symmetric copula families are more similar to the boundary copula $(C_b)_{b\in\R\setminus\{0\}}$, whereas the asymmetric Marshall-Olkin copula with $\alpha_1=1$ and varying $\alpha_2$ is closer to the diagonal. It in particular is an example of an SI copula family where in general $\sqrt{\xi} \not\le \rho$. 
        Right: Comparison of rank correlations for well known SI copula families.
        The curves show the difference between Spearman's rho (solid lines) or Kendall's tau (dashed lines) and Chatterjee's xi, plotted against Chatterjee's xi.
        The solid blue curve represent the copula family $(C_b)_{b>0}$ defined in \eqref{eq:C_x}, which yields the maximum possible difference between Spearman's rho and Chatterjee's xi. 
    }
    \label{fig:differences_curves}
\end{figure}

Due to Theorem~\ref{thm:rho_ge_xi}, \(\xi\) does not exceed \(\rho\) under the SI property. As the following example shows, this is no longer true if the underlying copulas satisfy the weaker positive dependence concept PLOD. Recall that a bivariate copula \(C\) is said to be positive lower orthant dependent (PLOD) if \(C(u,v) \geq uv\) for all \((u,v)\in [0,1]^2.\)

\begin{example}[PLOD is not sufficient for \(\xi(C)\le \rho(C)\)]\label{exPLOD}~\\
    Let \(U\) be uniform on \((0,1)\) and let \(f\colon [0,1]\to [0,1]\) be defined by \(f(u) = u\) for \(u\in [0,1/4]\cup (3/4,1],\) \(f(u) = u+1/4\) for \(u\in (1/4.1/2],\) and \(f(u) = u-1/4\) for \(u\in (1/2,3/4].\) Then \(V:= f(U)\) is uniform on \([0,1]\) and \(C(u,v):=P(U\le u, V\le v)\geq uv\) for all \((u,v)\in [0,1]^2.\) Hence, \(C\) is PLOD. Further, \(\xi(C)=1\) because \(V\) perfectly depends on \(U.\) However, \(\rho(C) = 13/16.\)
\end{example}

Studying the exact region for pairs of measures of association is not new. For the most popular measures of concordance---Spearman's rho and Kendall's tau---the exact region has been determined in \cite{Durbin-Stuart-1951} and \cite{Schreyer-2017}. Combinations with various other measures of (weak) concordance such as Blomqvist's beta, Spearman's footrule or Gini's gamma have been established in \cite{Bukovsek-2022,Kokol-2023,Kokol-2024,Kokol-2021}; see \cite[Table 1]{Tschimpke-2025} for an overview.
In contrast, and to the best of our knowledge, the behavior of Chatterjee’s rank correlation with respect to classical concordance measures has not been analysed so far.
Such a comparison is particularly important, however, as Chatterjee’s rank correlation is a rather exceptional dependence measure, for which an interpretation of its values is still challenging for practitioners. 
In empirical studies it is observed that Chatterjee's rank correlation is ``on average \(45-50~\%\) smaller than Pearson's''; see \cite{sierra2025revised}. Similar findings have been reported in comparisons with Spearman’s rank correlation by \cite{Behnam-2022}.
Our results for SI copulas support these empirical studies; specifically, Theorem \ref{thm:rho_ge_xi} shows that \(\xi\) does not exceed \(|\rho|\) for SI and SD copulas.
Inequalities for positive dependencies and, in particular, for SI copulas are of significant interest and have been studied extensively in the literature; see, for example, the so-called Hutchinson-Lai conjecture concerning the relationship between Spearman's rho and Kendall's tau (\cite{Balakrishnan-2009,Munroe-2010,Hurlimann-2003}).

The rest of the paper is organised as follows. Section~\ref{sec2} recalls the basic dependence orderings and dependence concepts that we use in this paper. In Section~\ref{sec3}, we study the optimisation problems underlying our main results. In particular, we define the new copula family \((C_b)_{b\in {\mathbb{R}}\setminus \{0\}}\) that describes the boundary of the \(\xi\)-\(\rho\)-region due to Theorem~\ref{thexirho}.
In Section \ref{sec4}, we show that the new copula family has diagonal-band structure and we study various of its properties. 
The proof of Theorem~\ref{thexirho}, together with the proofs of all its auxiliary results
(Lemmas~\ref{charSIcop}--\ref{lemCbcop}, Propositions~\ref{remcopfam}--\ref{propcfe},
and Theorem~\ref{thexirhooptimisation}), is provided in Section~\ref{sec5}.
The proof of Theorem~\ref{thm:rho_ge_xi} is given in Section~\ref{sec6}.

\section{Basic dependence concepts}\label{sec2}

A function \(C\colon [0,1]^2 \to [0,1]\) is a (bivariate) copula if it is grounded (i.e., \(C(u,0) = 0 = C(0,u)\) for all \(u\)), has uniform marginals (i.e., \(C(u,1)=u=C(1,u)\) for all \(u\)), and is two-increasing (i.e., \(C(u,v)+C(u',v')-C(u,v')-C(u',v)\geq 0\) for all \(u\le u'\) and \(v\le v'\)). Due to Sklar's theorem, every continuous bivariate distribution function \(F\) can uniquely be decomposed into its marginal distribution functions \(F_1\) and \(F_2\) and a copula \(C\) such that
\begin{align}\label{eqSklar}
    F(x,y) = C(F_1(x),F_2(y)), \quad (x,y)\in {\mathbb{R}}^2.
\end{align}
Vice versa, for every copula \(C\) and for all univariate distribution functions \(F_1\) and \(F_2,\) the function \(F\) in \eqref{eqSklar} defines a bivariate distribution function. Hence, the copula \(C\) models the dependence structure of \(F.\) Since Chatterjee's xi and Spearman's rho are invariant under strictly increasing transformations, they depend only on the underlying copula. We refer to \cite{Durante-2016} and \cite{Nelsen-2006} for an overview of the concept of copulas.

The proof of Theorem~\ref{thexirho} is based on solving Optimisation Problem \ref{optprob} (see Section \ref{sec3}), which makes use of the following lemmas.

\begin{lemma}[A characterisation of copulas]\label{charSIcop}
    A function \(C\colon [0,1]^2 \to [0,1]\) is a bivariate copula if and only if there exists a family \((h_v)_{v\in [0,1]}\) of measurable functions \(h_v\colon [0,1]\to [0,1]\) such that
    \begin{enumerate}[label=(\roman*), font=\upshape]
        \item \label{charSIcop1} \(C(u,v) = \int_0^u h_v(t) \de t\) for all \((u,v)\in [0,1]^2,\)
        \item \label{charSIcop3} \(\int_0^1 h_v(t) \de t = v\) for all \(v\in [0,1],\)
        \item \label{charSIcop2} \(h_v(t)\) is increasing in \(v\) for all \(t\in [0,1].\)
    \end{enumerate}
\end{lemma}

For bivariate copulas \(D\) and \(E,\) the lower orthant order \(D\leq_{lo} E\) is defined by \(D(u,v)\leq E(u,v)\) for all \((u,v)\in [0,1]^2.\) As a consequence of \eqref{frm_rho_integral}, Spearman's rho generally increases with the \(\leq_{lo}\)-order of copulas, while Chatterjee's xi is \(\leq_{lo}\)-increasing on the class of SI copulas as follows.

\begin{lemma}\label{lemSIcops}
    For SI copulas \(C,D\) with \(C\leq_{lo} D,\) it is \(\xi(C)\leq \xi(D)\) and \(\rho(C)\leq \rho(D).\)
\end{lemma}

For a copula \(C\in \mathcal{C},\) we refer to the uniquely determined SI (SD) copula \(C^\uparrow\) (\(C^\downarrow\)) constructed by rearrangements due to \cite[Proposition 3.17]{Ansari-Rueschendorf-2021} as the increasing (decreasing) rearranged copula of \(C\).

\begin{figure}[tb]
\centering
\includegraphics[width=0.38\textwidth, trim=0mm 10mm 0mm 10mm, clip]{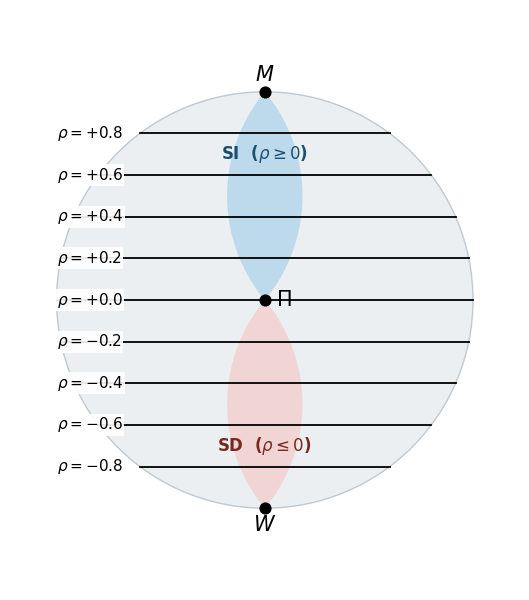}
\includegraphics[width=0.38\textwidth, trim=0mm 10mm 0mm 10mm, clip]{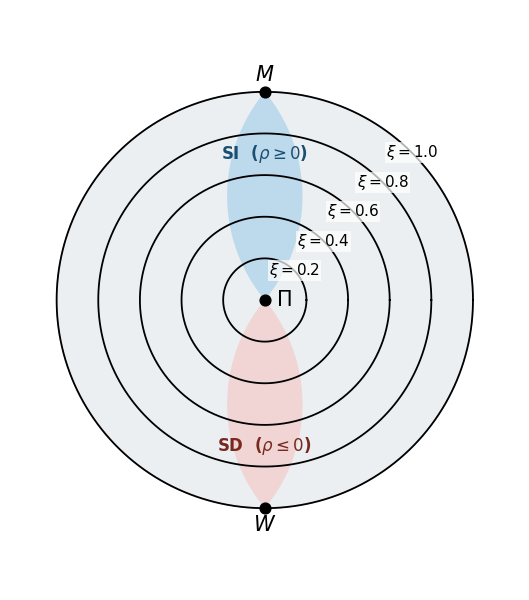}
\caption{
The class of bivariate copulas with level sets of Spearman’s $\rho$ (left) and Chatterjee’s $\xi$ (right) around the independence copula $\Pi.$
Constant values of $\rho$ appear as horizontal stripes; the extremes $\rho=\pm1$ occur only at the Fréchet copulas $M$ and $W.$  
In contrast, constant values of $\xi$ appear as concentric circles centred at $\Pi,$ expanding from $\xi=0$ (independence) to $\xi=1$ (perfect directed dependence).  
The blue and red lenses highlight stochastically increasing (SI) and decreasing (SD) copulas, respectively, which form convex subsets of the set of all copulas $\mathcal{C}.$
\label{fig:xi_and_rho}
}
\end{figure}

\begin{figure}[t!]
    \centering
    \includegraphics[width=0.9\textwidth, trim={40 100 0 20}, clip]{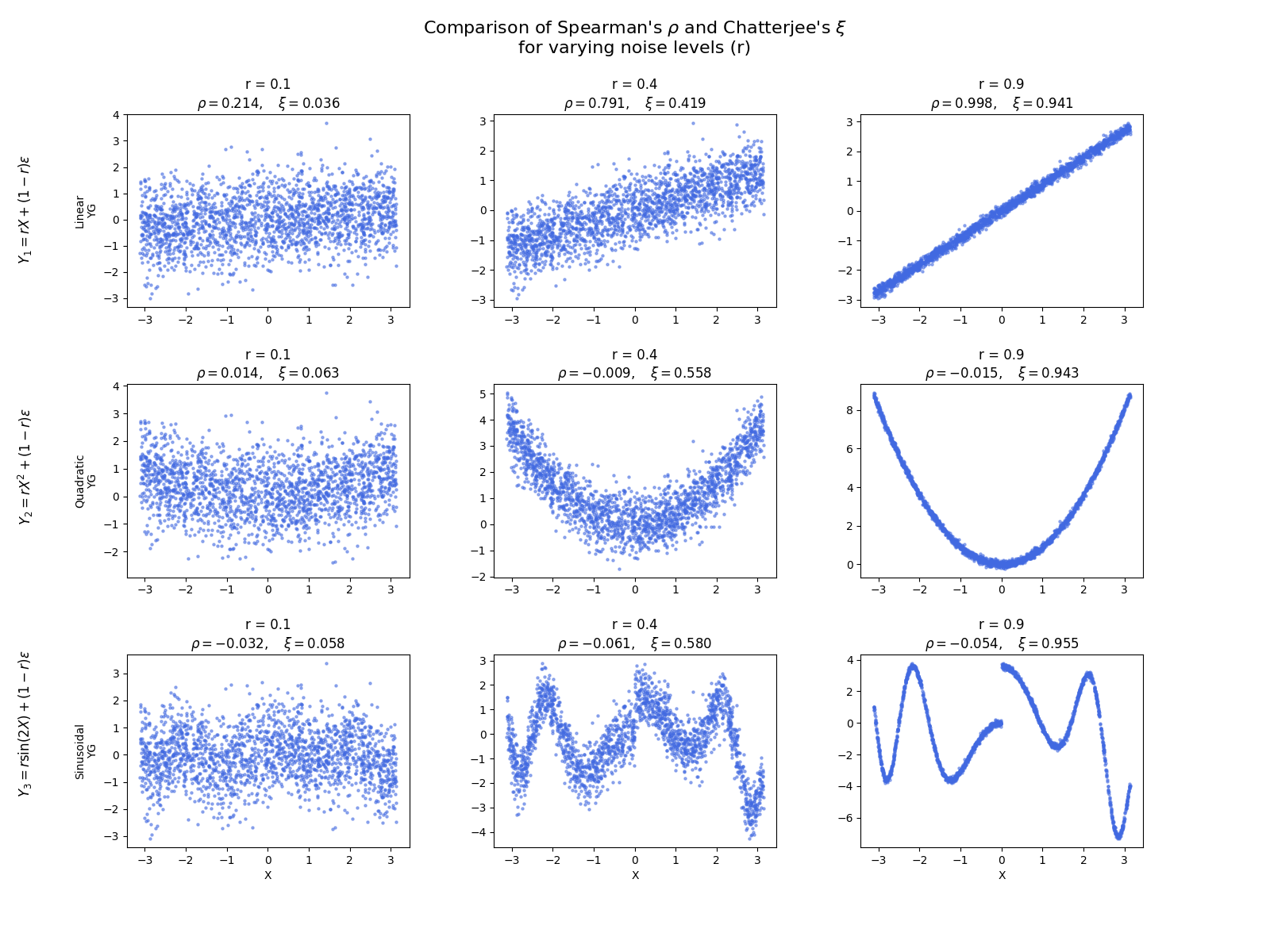}
    \caption{Comparison of Spearman's rank correlation $\rho$ and Chatterjee's rank correlation $\xi$ across different functional dependencies and noise levels. The first row corresponds to the linear model $Y_1 = r X + (1-r) \varepsilon$, the second row to the quadratic model $Y_2 = r X^2 + (1-r) \varepsilon$, and the third row to the sinusoidal model $Y_3 = r(4\sin(-X^2) + (2-X)(2+X)\1_{\{X>0\}}) + (1-r) \varepsilon$, where \(X\) is uniform on \((-\pi,\pi)\) and independent from the standard normal error \(\varepsilon\).
    The columns correspond to an increasing signal parameter $r \in \{0.1, 0.4, 0.9\}$ and decreasing noise. The figure illustrates that $\rho$ effectively detects the strength of monotonic (in particular, linear) dependence (row 1). For non-monotonic dependence, it may attain the value zero as illustrated in rows 2 and 3. In contrast, $\xi$ increases with $r$ across all models, reflecting the strength of the functional dependence regardless of monotonicity.}
    \label{fig:rho_xi_comparison}
\end{figure}
\newpage
\begin{lemma}\label{lemrearrangedcop}
    The rearranged copulas \(C^\uparrow\) and \(C^\downarrow\) of \(C\) satisfy
    \begin{enumerate}[label=(\roman*), font=\upshape]
        \item \label{lemrearrangedcop1}\(C^\downarrow (u,v)\leq C(u,v) \leq C^\uparrow(u,v)\) and \(C^\downarrow(u,v) = v - C^\uparrow(1-u,v)\),  \((u,v)\in [0,1]^2,\)
        \item \label{lemrearrangedcop2} \(\xi(C^\downarrow) = \xi(C) = \xi(C^\uparrow),~\rho(C^\downarrow) \le \rho(C) \le \rho(C^\uparrow)\), and \(\rho(C^\downarrow) = -\rho(C^\uparrow).\)
    \end{enumerate}
\end{lemma}

We also use the following properties of Spearman's and Chatterjee's rank correlation under convex combinations of copulas.

\begin{lemma}\label{lem:convexity_xi}
    Let $C_1,\,C_2$ be copulas and let $\lambda \in [0,1].$
    If
    \( 
        C \;=\; \lambda\,C_1 \;+\; (1-\lambda)\,C_2
    ,\)
    then we have
\begin{enumerate}[label=(\roman*), font=\upshape]
    \item \(\rho(C) = \lambda \rho(C_1) + (1-\lambda)\rho(C_2)\),
    \item \(\xi(C)\le  \lambda\,\xi(C_1) + (1-\lambda)\,\xi(C_2).\)
\end{enumerate}
\end{lemma}

\section{A convex optimisation problem}
\label{sec3}

The proofs of Theorem~\ref{thexirho} and Theorem~\ref{thm:rho_ge_xi} are based on a representation of \(\rho(C) - \xi(C)\) in terms of a partial derivative of \(C.\)
To this end, we consider the function $t \mapsto h_v(t):=\partial_1 C(t,v)$ for \(v\in [0,1].\)
Recall that \(h_v\) maps from \([0,1]\) to \([0,1],\) is measurable and well-defined outside a Lebesgue-null set.
Using the disintegration $C(u,v)=\int_0^u h_v(t)\de t$ and applying Fubini's theorem, we have
\begin{align}\label{reprho2}
\int_{[0,1]^2}C(u,v)\de \lambda^2(u,v)
=\int_{0}^{1}\int_{0}^{1}\Bigl(\int_{0}^{u}h_v(t)\de t\Bigr)\de u \de v
=\int_{0}^{1}\int_{0}^{1}(1-t)h_v(t)\de t \de v.
\end{align}
Hence, one can rewrite the difference of Spearman's rho and Chatterjee's xi as
\begin{align}
 \label{optprob1a}
  \hspace{-.23cm}  \rho(C)-\xi(C) &= \left[12 \int_0^1\int_0^1 (1-t) h_v(t) \de t \de v - 3\right]  -\left[6\int_0^1\int_0^1 h_v(t)^2 \de t \de v - 2\right]\\
 \label{optprob1b} &= 
  6 \int_{0}^{1}\!\underbrace{\int_{0}^{1}\!\bigl(2(1-t)h_v(t)-h_v(t)^2\bigr)\de t}
    _{=:F_v(h_v)}\de v-1 .
\end{align}
First, let us consider the inner integral in \eqref{optprob1b} as a concave functional \(F_v\) in \(h_v.\) If \(t \mapsto h_v(t)\) is decreasing for all \(v\) or, equivalently, if \(C\) is an SI copula, we can show that the expression in \eqref{optprob1b} is non-negative. More precisely, we use a maximum principle for convex sets to prove that the minimum over all SI copulas is attained at the boundary of the set.
In this case, it follows that
\begin{align}
    F_v(h_v)\geq v(1-v)
\end{align}
with equality for all \(v\) if and only if \(C\) is the product copula \(\Pi\) or the upper Fr\'{e}chet copula \(M.\) This is the content of Theorem~\ref{thm:rho_ge_xi}.

For proving Theorem~\ref{thexirho}, we determine the maximal value of Spearman's rho for a fixed value \(x\in [0,1]\) of Chatterjee's rank correlation, i.e., we solve the maximisation problem 
\begin{align}\label{opt1a}
   M_x := \max\{\rho(C) \mid C\in \mathcal{C}, ~\xi(C) = x\}.
\end{align}
To determine \(M_x,\) the idea is to solve an extended maximisation problem using convexity of a modified constraint set.
To this end, we first observe that solutions to \eqref{opt1a} are SI copulas as a consequence of Lemma \ref{lemrearrangedcop}.
Second, we know from Lemma~\ref{lemSIcops} that, for SI copulas, both \(\xi\) and \(\rho\) are increasing with the lower orthant order.
Then, using a simple rearrangement argument, one can show that the maximisation problem \eqref{opt1a} is equivalent to
\begin{align}\label{opt1}
   M_x = \max\{\rho(C) \mid C\in \mathcal{C}, ~\xi(C) \le x\}.
\end{align}

Using the representation \eqref{optprob1a} and the characterisation of copulas in Lemma~\ref{charSIcop}, we can now rewrite the extended maximisation problem \eqref{opt1} as follows.

\begin{optimisationproblem}\label{optprob}
Over all measurable functions \((t,v)\mapsto h_v(t),\)
\begin{align}
    \text{maximise} \quad &\int_0^1\int_0^1 (1-t) h_v(t) \de t \de v \nonumber\\
    \label{eqthexirho1a} \text{subject to} \quad  &0\le h_v(t) \le 1 \text{ for all } v,t\in [0,1],\\
    \label{eqthexirho1b}  \quad  &h_v(t) \le h_{v'}(t)  \text{ for all } t,v,v'\in [0,1] \text{ with } v\le v',\\
    \label{eqthexirho2} \quad  &\int_0^1 h_v(t) \de t = v \quad \text{for all } v\in [0,1],\\
    \label{eqthexirho3} \text{and} \quad & 6 \int_0^1\int_0^1 (h_v(t))^2 \de t \de v - 2 \le x. 
\end{align}
\end{optimisationproblem}
While the constraints in \eqref{eqthexirho1a} -- \eqref{eqthexirho2} ensure that the family \((h_v)_{v\in [0,1]}\) yields a copula \(C,\) the constraint in \eqref{eqthexirho3} means that \(\xi(C) \le x.\) 
The latter inequality constraint is more convenient than the equality constraint $\xi(C)=x,$ because it yields a convex optimisation problem where the objective function is linear and the set of constraints is convex.
We solve this optimisation problem in Theorem \ref{thexirhooptimisation}, where we first show that the family \(h^* = (h_v)_{v\in [0,1]}\) of clamped linear functions in \eqref{h_v_clamped} is a stationary point of the (generalised) Lagrangian in \eqref{eq:generalised_lagrangian}.
Then, we use the quadratic structure of the Lagrangian and verify by \eqref{seuffsecordcond} the sufficient second-order conditions in \cite[Thm.~3.63]{bonnans2013perturbation} to show that the stationary point \(h^*\) is the unique local optimum.
Using the fact that the objective function is linear and the constraint set is convex, it follows that \(h^*\) is the unique global optimum.
Finally, applying Lemma~\ref{charSIcop}, we obtain a unique copula that solves Optimisation Problem \ref{optprob} and thus the maximisation problem \eqref{opt1a}.

To this end, consider for $b>0$ and $(u,v)\in[0,1]^2$ the function \(C_b\) defined by
\begin{align}\label{eq:C_x}
    C_b(u,v) = \begin{cases}
        u - \frac{b}{2} (u-a_v)^2 + \frac{b}{2} (a_v\wedge 0)^2 & \text{if } a_v < u \leq s_v,\\
        \min\{u,v\} & \text{else},
    \end{cases}
\end{align}
with \(s_v\) and \(a_v\) given by
\begin{align}
\label{eq:s_v}
  s_{v} &:=
\begin{cases}
\displaystyle
\sqrt{\tfrac{2v}b},
& \text{if } v \le \frac{1}{2b}\wedge\frac{b}{2}, \\
v+\frac{1}{2b}, & \text{if } \frac{1}{2b} < v \le 1-\frac{1}{2b}, \\
\frac{v}b+\frac{1}{2},
& \text{if } \frac{b}{2} < v \le 1-\frac{b}{2}, \\
1+\tfrac1b-\sqrt{\tfrac{2(1-v)}b},
& \text{if } v > 1 - \bigl(\frac{1}{2b} \wedge \frac{b}{2}\bigr),
\end{cases}
\qquad
a_v := s_v - \tfrac 1b,
\end{align}
where \(\wedge\) denotes the minimum of two real numbers (later we will also use the symbol \(\vee\) for the maximum).
The dependence curves \(v\mapsto s_v\) and \(v\mapsto a_v\) are functions of the copula dependence parameter \(b\). They describe the boundary of the support of \(C_b\); see Equation \eqref{eq:density_cb} and Fig.~\ref{fig:C_x}.
Furthermore, for $b<0,$ define
\begin{align}\label{defcopbneg}
     C_b(u,v) := v - C_{-b}(1-u,v), \quad (u,v)\in [0,1]^2.
\end{align}
The following result shows that \((C_b)_{b\in {\mathbb{R}}\setminus\{0\}}\) is a family of bivariate copulas. We refer to Section \ref{sec4} for a detailed discussion of their properties.

\begin{lemma}\label{lemCbcop}
    The function \(C_b\) defined by \eqref{eq:C_x} for \(b>0\) and by \eqref{defcopbneg} for \(b<0\) is a bivariate copula. 
\end{lemma}

Due to the following theorem, which underlies Theorem~\ref{thexirho}, the copulas \((C_b)_{b\in\R\setminus\{0\}}\) describe the boundary of the \(\xi\)-\(\rho\)-region.
More precisely, for each \(x\in (0,1),\) we determine the copula parameter \(b=b_x\) and show that \(C_b\) is the unique solution to \eqref{opt1a}.
The proof is based on solving convex Optimisation Problem \ref{optprob} as discussed above.

\begin{theorem}[Solution to maximisation problem \eqref{opt1a}]\label{thexirhooptimisation}~\\
For \(b>0,\) the copula \(C_b\) in \eqref{eq:C_x} has the following properties:
\begin{enumerate}[label=(\roman*), font=\upshape]
    \item If \(x\in (0,1),\) then for \(b = b_x\) in \eqref{eq:b_and_M}, the copula \(C_b\) is the unique solution to Optimisation Problem \eqref{opt1a}. In particular, we have \(\xi(C_{b}) = x\) and \(\rho(C_{b}) = M_x.\)
    \item \label{thexirhooptimisation2} If \(x = 0,\) then the product copula \(\Pi\) is the unique solution to \eqref{opt1a}.
    \item \label{thexirhooptimisation3} If \(x=1,\) then the upper Fr\'{e}chet copula \(M\) is the unique solution to \eqref{opt1a}.
\end{enumerate}
\end{theorem}

Some symmetry arguments show that \(C_b\) solves for \(b<0\) the minimization problem associated with \eqref{opt1a}; see also Remark \ref{remsymm}\,\ref{remsymma}.

\section{Properties of the new diagonal band copula family}\label{sec4}

In this section, we study various properties of the new copula family \((C_b)_{b\in \R}\) defined by \eqref{eq:C_x} for \(b>0\) and by \eqref{defcopbneg} for \(b<0\).

The following result states that the first partial derivative of \(C_b\) is of clamped linear form; see Fig.~\ref{fig:C_x}. Note that
the family \((h_v)_{v\in [0,1]}\) in \eqref{h_v_clamped} solves Optimisation Problem \ref{optprob}; see the proof of Theorem \ref{thexirhooptimisation}.

\begin{proposition}[Clamped linear form of \(\partial_1 C_b\)]\label{remcopfam}~
    For \(b>0\), the derivative of \(C_b\) with respect to its first component is given by
\begin{align}\label{h_v_clamped}
   h_v(t)=\partial_1 C_b(t,v) = \operatorname{clamp}\left(b(s_v-t),\,0,\,1\right), \quad (t,v)\in [0,1]^2,
\end{align}  
where \(\operatorname{clamp}(y,\,\alpha,\,\beta) := (y \vee \alpha) \wedge \beta\).
\end{proposition}

Differentiating \(h_v\) in \eqref{h_v_clamped} with respect to \(v\) shows that \(C_b\) admits a Lebesgue density which is constant in \(u\) on the support of \(C_b\); see Fig. \ref{fig:C_x}. In the next result, we also show that \(C_b\) has convex support.

\begin{proposition}[Absolute continuity of \(C_b\)]\label{propabscont}
    For all \(b>0\), the copula \(C_b\) admits a Lebesgue density given by
    \begin{align}\label{eq:density_cb}
        c_b(u,v) = \begin{cases}
            b\,s_v', & \text{if } a_v < u < s_v, \\
            0, & \text{otherwise},
        \end{cases}
    \end{align}
    where \(s_v':= d s_v / dv\).
    Specifically, the support \(\supp(C_b) = \overline{\{(u,v) \in [0,1]^2 : c_b(u,v) > 0\}}\) is a convex set, and for \((u,v)\in \supp(C_b)\), we have
    \begin{align}\label{eqpropabscont}
        c_b(u,v) = \begin{cases}
            1/s_v, & \text{if } v \le \frac{1}{2b}\wedge\frac{b}{2}, \\
            b\vee 1, & \text{if } \frac{1}{2b}\wedge\frac{b}{2} < v \le 1 - (\frac{1}{2b}\wedge\frac{b}{2}), \\
            1/(1-a_v), & \text{if } v > 1 - (\frac{1}{2b}\wedge\frac{b}{2}).
        \end{cases}
    \end{align}
\end{proposition}

The copula family \((C_b)_{b\in \R\setminus \{0\}}\) is one-parametric with dependence parameter \(b\). The dependence curves \(v\mapsto a_v\) and \(v\mapsto s_v\) depend on \(b\) and describe the diagonal-band structure of \(C_b\) as we discuss in the following remark.

\begin{remark}[Interpreting the copula parameter]
    \leavevmode
    \begin{enumerate}[label=(\alph*)]
    \item The dependence parameter \(b\) of the copula \(C_b\) describes the negative slope of the linear copula derivative \(t\mapsto h_v(t) = \partial_1 C_b(t,v)\) for \(a_v \vee 0 < t < s_v \wedge 1;\) see \eqref{h_v_clamped} and the right panel in Fig.~\ref{fig:C_x}. 
        \item The dependence curves \(v\mapsto s_v\) and \(v\mapsto a_v,\) defined in \eqref{eq:s_v}, are fully determined by \(b.\)
The value \(s_v\) is the root of the linear part \(t\mapsto b(s_v-t)\) of \(h_v\) in \eqref{h_v_clamped}.
Similarly, \(a_v\) is the solution to \(b(s_v-t) = 1\) as an equation in \(t.\)
Both values can be interpreted geometrically in the following sense:
If we consider \(h_v\) in \eqref{h_v_clamped} as a function of \(t\) not just on $[0,1]$ but naturally extended to \({\mathbb{R}},\) and if we let
\(
  S_v := \{t\in{\mathbb{R}} \mid 0 < h_v(t) < 1\} 
,\)
then it holds that $s_v = \sup S_v$ and $a_v = \inf S_v.$
Visualising the density $c_b$ of the copula \(C_b\) in Fig.~\ref{fig:C_x}, we see that the set
\begin{align}\label{eq:diag_band}
  S:=\{(t,v) \mid \; t\in S_v \text{ and } v\in[0,1]\}
\end{align}
is a diagonal band, which, restricted to the square $[0,1]^2,$ yields the support of $c_b.$
\item For \(b>1\), the density attains the constant value \(c_b(u,v) = b\) for all \((u,v)\in (a_v,s_v)\times (\tfrac 1 {2b}, 1-\tfrac 1 {2b}),\) while, for \(0  < b<1,\) the density attains the constant value \(c_b(u,v) = 1\) for all \((u,v)\in (0,1)\times (\tfrac b 2, 1-\tfrac b2)\,;\) the respective cases are visualised by the area between the two dotted black lines drawn for $b=0.5$ and $b=5$ in Fig.~\ref{fig:C_x}.
Note that \((C_b)_{b\in \mathbb{R}\setminus \{0\}}\) is an asymmetric diagonal band copula family.
Symmetric diagonal band copulas are studied in \cite{cooke1986monte} and \cite{Kotz-2010}.
\end{enumerate}
\end{remark}

The copula family \((C_b)_{b\in \R\setminus \{0\}}\) admits the following limiting cases.

\begin{proposition}[Limiting cases]\label{proplimcase}
    We have
    \begin{enumerate}[label=(\roman*), font=\upshape]
        \item \label{proplimcase1} \(C_b \to \Pi\) for \(b\to 0\),
        \item \label{proplimcase2} \(C_b \to M\) for \(b\to \infty\),
        \item \label{proplimcase3} \(C_b\to W\) for \(b\to -\infty\),
    \end{enumerate}
    where each convergence is uniform.
\end{proposition}
 
\begin{figure}[htpb]
 \centering 
 \includegraphics[width=0.5\textwidth, trim= 0mm 3mm 00mm 0mm, clip]{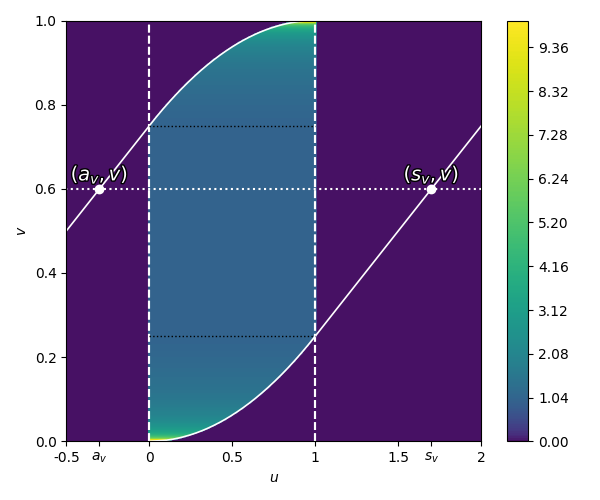}\label{fig:rho_minus_xi_maximal_copula_0.5}
 \hfill
 \includegraphics[width=0.4\textwidth, trim= 0mm 0mm 00mm 0mm, clip]{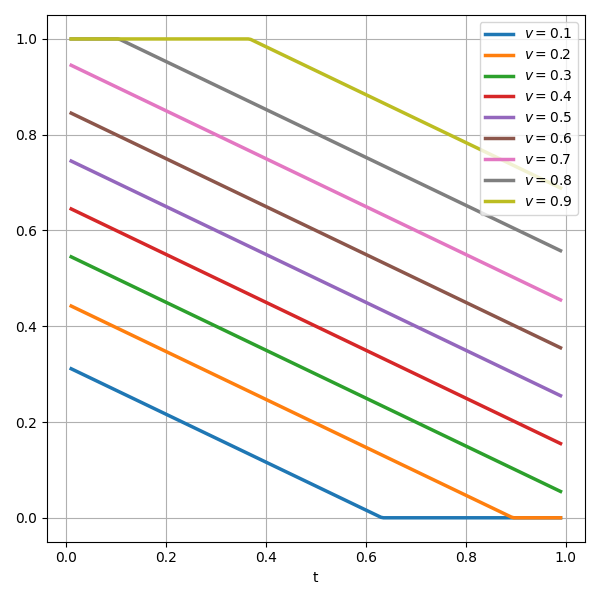}\label{fig:cond_distr_b_0.5} \\[-0.6ex]

 \includegraphics[width=0.5\textwidth, trim= 0mm 3mm 00mm 0mm, clip]{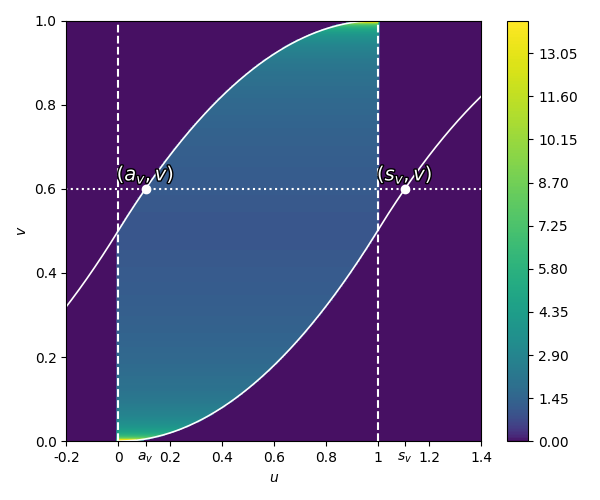}\label{fig:rho_minus_xi_maximal_copula_1}
 \hfill
 \includegraphics[width=0.4\textwidth, trim= 0mm 0mm 00mm 0mm, clip]{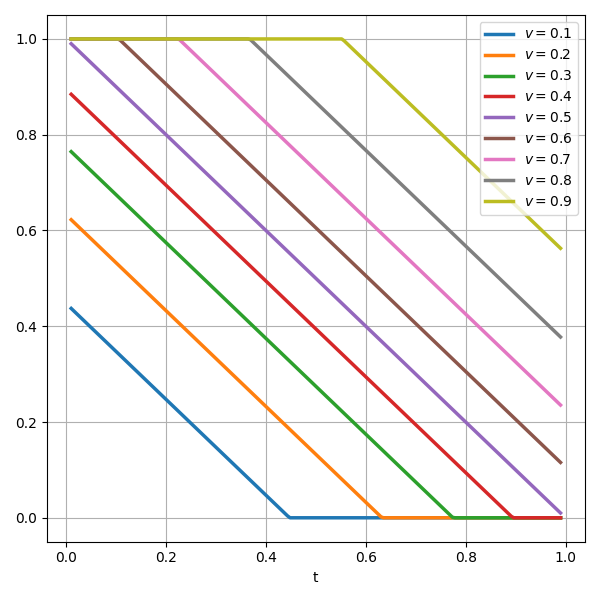}\label{fig:cond_distr_b_1} \\[-0.6ex]

 \includegraphics[width=0.5\textwidth, trim= 0mm 3mm 00mm 0mm, clip]{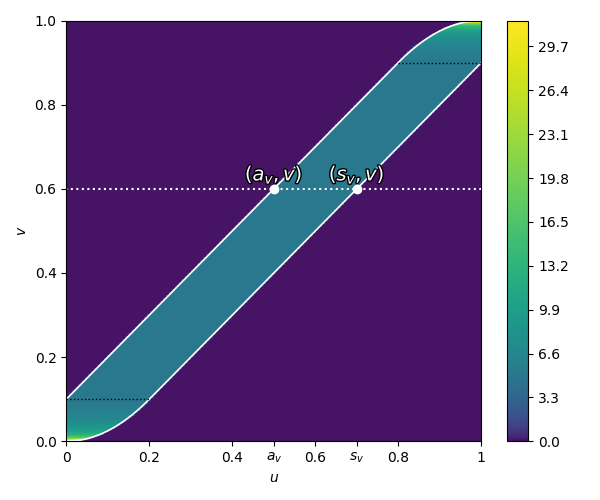}\label{fig:rho_minus_xi_maximal_copula_5}
 \hfill
 \includegraphics[width=0.4\textwidth, trim= 0mm 0mm 00mm 0mm, clip]{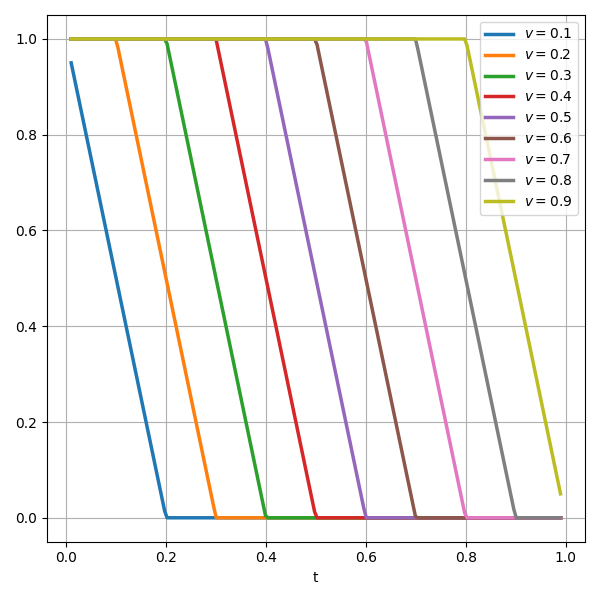}\label{fig:cond_distr_b_5}

 \caption{The density \(c_b\) (left) and the derivative \(t\mapsto h_v(t) = \partial_1 C_b(t,v)\) (right) of the copula \(C_b\) for $b=0.5$ (top), $b=1$ (middle) and $b=5$ (bottom). \(s_v\) is the (hypothetical) upper boundary of the band and \(a_v\) is the (hypothetical) lower boundary of the band, visualised in the densities for $v=0.6.$ The support of the densities on the left is just the intersection of the diagonal band \(S\) in \eqref{eq:diag_band} and the unit square $[0,1]^2$. Note that the density is zero outside the band; see \eqref{eq:density_cb}.}
 \label{fig:C_x}
\end{figure}

The next result shows that the copula family \((C_b)_{b\in \R\setminus \{0\}}\) ranges from perfect negative dependence via independence to perfect positive dependence. Therefore, recall that \(\leq_{lo}\) denotes the pointwise order of copulas. Further, we observe from \eqref{h_v_clamped} that, for \(b>0,\) the copula derivative \(h_v\) is decreasing for all \(v\in [0,1]\). Consequently, \(C_b\) is an SI copula for \(b>0\). As we show, \(C_b\) even satisfies the stronger \(\mathrm{MTP}_2\)-property. This positive dependence concept is defined for multivariate distributions that admit a log-supermodular Lebesgue density \(f\), i.e. the density satifies \(f(x) f(y) \leq f(x\wedge y) f(x\vee y)\) for all \(x,y\). Here \(\wedge\) and \(\vee\) denote the componentwise minimum and maximum, respectively.

\begin{proposition}[Dependence properties]\label{propdepprop}
We have
\begin{enumerate}[label=(\roman*), font=\upshape]
    \item \label{propdepprop1}  \(C_b\leq_{lo} C_{b'}\) for \(b, b'\in \R\setminus\{0\}\) with \(b \leq b'\).
    \item \label{propdepprop2} Positive dependence: For \(b>0\), \(C_b\) is \(\mathrm{MTP}_2\) and thus SI.
\end{enumerate}
\end{proposition}

In the following remark, we discuss several symmetry properties of \(C_b\).

\begin{remark}[Symmetry properties]\label{remsymm}
\leavevmode
    \begin{enumerate}[label=(\alph*)]
        \item \label{remsymma} For \(b<0,\) \(C_b\) is the decreasing rearranged copula of $C_{-b},$ i.e., \(C_b = C_{-b}^\downarrow;\) see Lemma \ref{lemrearrangedcop} \ref{lemrearrangedcop1}. By symmetry, \(C_b\) exhibits analogous properties for negative values of \(b\) as for positive values, as described in Theorem~\ref{thexirhooptimisation} and Proposition \ref{propdepprop}. For example, \(\xi(C_b) = \xi(C_{-b})\) and \(\rho(C_b) = - \rho(C_{-b}).\) In particular, \(C_b\) solves for \(b < 0\) the minimisation problem associated with \eqref{opt1a}.
        \item The copula \(C_b\) satisfies several symmetry and invariance properties. First, note that the dependence curves \(a_v\) and \(s_v\) are related by \(a_{1-v} = 1-s_v\), which follows from \(s_v + s_{1-v} = 1 + 1/b\) for all \(v\in [0,1]\) and \(b\ne 0\). This implies that the copula \(C_b\) is radially symmetric about \((1/2,1/2)\), i.e., the copula density satisfies \(c_b(u,v) = c_b(1-u,1-v)\) for all \((u,v)\in [0,1]^2\); see Fig.~\ref{fig:C_x}. As a consequence of radial symmetry, we obtain that, for \(b<0\), the copula in \eqref{defcopbneg} can also be written as \(C_b(u,v) = u - C_{-b}(u,1-v)\). Lastly note that \(C_b\) is positive (negative) measure-inducing in the sense of \cite{Fuchs-Tschimpke-2024} for \(b>0\) (\(b<0\)). This is also a consequence of radial symmetry.
        \item For $b<0$, in analogy to Proposition \ref{remcopfam}, radial symmetry yields
            \[
                h_v(t)=\partial_1 C_b(t,v) = \operatorname{clamp}\left(b(s^*_v-t),\,0,\,1\right),
            \]
            where $s^*_v := 1 - s_v$ with \(-b\) inserted to $s_v$ in \eqref{eq:s_v}.
            Likewise, for \(a_v^* = 1- a_v\), we obtain the density 
            \[
            c_b(u,v) = \begin{cases}
                b\,(1-s_v^*)', & \text{if } s_v^* < u < a_v^*,\\
                0, & \text{otherwise}.
            \end{cases}
            \]
    \end{enumerate}
\end{remark}

As we show in the next result, for the copula family \((C_b)_{b\in \R\setminus 0}\), there are simple formulas available for Chatterjee's rank correlation as well as for Spearman's and Kendall's rank correlation.
Closed-form expressions of these three rank correlations are also available for some other copula families such as for Gaussian copulas. In general, however, closed-form expressions are rare; see \cite[Table 6]{Ansari-Rockel-2023} for an overview.
We define $\mathrm{sign}(b) := 1$ for $b > 0$ and $\mathrm{sign}(b) := -1$ for $b < 0$.

\begin{proposition}[Rank correlations]\label{propcfe}
For $b \in \mathbb{R} \setminus \{0\}$, the copula $C_b$ admits closed-form expressions for
\begin{enumerate}[label=(\roman*), font=\upshape]
    \item Chatterjee's rank correlation:
    \[
        \xi(C_b) =
        \begin{cases}
            \dfrac{b^2}{10}(5-2|b|), & |b| \le 1, \\[2.5ex]
            1 - \dfrac{1}{|b|} + \dfrac{3}{10b^2}, & |b| \ge 1,
        \end{cases}
    \]

    \item Spearman's rank correlation:
    \[
        \rho(C_b) = \mathrm{sign}(b)
        \begin{cases}
            |b| - \dfrac{3b^2}{10}, & |b| \le 1, \\[2.5ex]
            1 - \dfrac{1}{2b^2} + \dfrac{1}{5|b|^3}, & |b| \ge 1,
        \end{cases}
    \]
    
    \item Kendall's rank correlation:
    \[
        \tau(C_b) = \mathrm{sign}(b)
        \begin{cases}
            \dfrac{2|b|}{3} - \dfrac{b^2}{6}, & |b| \le 1, \\[2.5ex]
            1 - \dfrac{2}{3|b|} + \dfrac{1}{6b^2}, & |b| \ge 1.
        \end{cases}
    \]
\end{enumerate}
\end{proposition}

Recall Corollary~\ref{cor:max-rho-minus-xi}, where we showed that \(\xi(C_1)=0.3\) and \(\rho(C_1)=0.7\).
Interestingly, Kendall’s tau of the copula \(C_1\) also attains a one-decimal value, namely \(\tau(C_1)=0.5\).
The copula \(C_1\) is visualized in the middle row of Fig.~\ref{fig:C_x}.

\section{Proof of Theorem~\ref{thexirho}}\label{sec5}

The proof of our first main result, Theorem~\ref{thexirho}, is based on several lemmas and propositions as well as on Theorem~\ref{thexirhooptimisation} whose proofs we give first in the following subsections.

\subsection{Proofs of Lemmas \ref{charSIcop}--\ref{lemCbcop}}

\begin{proof}[Proof of Lemma~\ref{charSIcop}.]
    '\(\Longrightarrow\)': For \((U,V)\sim C,\) choose \(h_v\) as the conditional distribution function \(t\mapsto F_{V|U=t}(v).\) Then \(h_v\) satisfies properties \ref{charSIcop1} -- \ref{charSIcop2}.\\
    '\(\Longleftarrow\)': It is straightforward to show that \((u,v)\mapsto C(u,v) = \int_0^u h_v(t) \de t\) is grounded, \(2\)-increasing, and has uniform marginals. 
\end{proof}

\begin{proof}[Proof of Lemma \ref{lemSIcops}.]
As a consequence of \cite[Lemmas 2.4 and 2.10]{Ansari-Rockel-2023}, \(\xi\) is increasing with respect to the \(\leq_{\partial_1 S}\)-ordering, which is for SI copulas by \cite[Lemmas 2.6 (ii)]{Ansari-Rockel-2023} equivalent with the \(\leq_{lo}\)-ordering. 
\end{proof}

\begin{proof}[Proof of Lemma \ref{lemrearrangedcop}]
    Statement \ref{lemrearrangedcop1} is given in \cite[Proposition 3.17]{Ansari-Rueschendorf-2021}.
    The invariance property for \(\xi\) in \ref{lemrearrangedcop2} follows from \cite[Lemma 2.10]{Ansari-Rockel-2023} using that the rearranged copulas are constructed through the Schur order in \cite[Definition 2.3]{Ansari-Rockel-2023}.
\end{proof}

\begin{proof}[Proof of Lemma \ref{lem:convexity_xi}]
    (i) This is a direct consequence of the representation of \(\rho(C)\) in \eqref{frm_rho_integral}.

    (ii) Recall that \(\xi(C) = 6 \int_{[0,1]^2} (\partial_1 C(u,v))^2 \de u \de v - 2\). Since differentiation is a linear operator, we have \(\partial_1 C = \lambda \partial_1 C_1 + (1-\lambda) \partial_1 C_2\) for $\lambda\in[0,1]$.
    The function \(f(x) = x^2\) is convex. Therefore, by Jensen's inequality,
    \[
        (\partial_1 C)^2 = (\lambda \partial_1 C_1 + (1-\lambda) \partial_1 C_2)^2 \le \lambda (\partial_1 C_1)^2 + (1-\lambda) (\partial_1 C_2)^2.
    \]
    Integrating over \([0,1]^2\) yields
    \begin{align*}
        \xi(C)
        &\le \lambda \left(6 \int_0^1\int_0^1 (\partial_1 C_1)^2 \de u \de v\right) + (1-\lambda) \left(6 \int_0^1\int_0^1 (\partial_1 C_2)^2 \de u \de v\right) - 2 
        = \lambda \xi(C_1) + (1-\lambda) \xi(C_2),
    \end{align*}
    which proves the claim.
\end{proof}

\begin{proof}[Proof of Lemma \ref{lemCbcop}]
We distinguish the cases \(b>0\) and \(b<0\).

\emph{Case \(b>0\):}
Let \(h_v(t) := \partial_1 C_b(t,v)\). 
By \eqref{eq:C_x}, one has \(h_v(t) = b(s_v - t)\) for all \(t\in (a_v, s_v)\).
Furthermore, it is straightforward to check that $a_v \le v \le s_v$ for all \(v\in [0,1]\), so that \eqref{eq:C_x} yields $h_v(t) = 1$ for $t < a_v$ and $h_v(t) = 0$ for $t > s_v$. In total, we have
\begin{equation}\label{defhvhv}
    h_v(t) = \operatorname{clamp}\bigl(b(s_v-t),\, 0,\, 1\bigr),
\end{equation}
for all \(t\in (0,1)\setminus\{a_v,s_v\}\), where \(\operatorname{clamp}(y,\,\alpha,\,\beta) = (y \vee \alpha) \wedge \beta\).
To show that \(C_b\) is a copula, we verify the conditions of Lemma \ref{charSIcop}.
First, note that \(h_v\) maps to \([0,1]\) for all \(t,v\in [0,1]\).
The remaining conditions are easily verified as follows:
\begin{enumerate}[label=(\roman*), font=\upshape]
    \item By the fundamental theorem of calculus and \(C_b(0,v)=0\), we have \(C_b(u,v) = \int_0^u h_v(t) \de t\).
    \item Differentiating the expressions in \eqref{eq:s_v} shows that \(v \mapsto s_v\) is strictly increasing. Since \(b>0\), the function \(s \mapsto \operatorname{clamp}(b(s-t), 0, 1)\) is increasing. Thus, \(h_v(t)\) is increasing in \(v\).
    \item We verify that \(\int_0^1 h_v(t) \de t = v\) by calculating the integral explicitly for the four cases in \eqref{eq:s_v}.
    \begin{enumerate}[label=\arabic*., font=\upshape]
        \item If \(s_v \le 1\) and \(a_v \le 0\), the support of \(h_v\) on \([0,1]\) is \([0, s_v]\). The condition \(a_v \le 0\) is equivalent to \(s_v \le 1/b\), so it must be \(s_v \le 1 \wedge 1/b\).
        Using that the function \(v \mapsto s_v\) is strictly increasing, we evaluate \(s_v\) at the upper bound of the first case, \(v^* = \frac{1}{2b} \wedge \frac{b}{2}\), and obtain
        \[
            s_{v^*} = \sqrt{\frac{2v^*}{b}} = \sqrt{\frac{2}{b} \left(\frac{1}{2b} \wedge \frac{b}{2}\right)} = \sqrt{\frac{1}{b^2} \wedge 1} = \frac{1}{b} \wedge 1.
        \]
        Thus, the condition \(s_v \le 1 \wedge 1/b\) holds if and only if \(v \le \frac{1}{2b} \wedge \frac{b}{2}\). This confirms we are in the first case of \eqref{eq:s_v}, where \(s_v = \sqrt{2v/b}\). Hence,
        \[
            \int_0^1 h_v(t)\de t = \int_0^{s_v} b(s_v-t) \de t = \frac{b}{2}s_v^2 = \frac{b}{2}\left(\frac{2v}{b}\right) = v.
        \]
        \item If \(s_v \le 1\) and \( a_v > 0\), the geometric constraints imply \(1/b < s_v \le 1\). This interval is non-empty only if \(b > 1\).
        We check the boundaries of the second case in \eqref{eq:s_v}: at the lower bound \(v = \frac{1}{2b}\), we have \(s_v = 1/b\), and at the upper bound \(v = 1 - \frac{1}{2b}\), we have \(s_v = 1 - \frac{1}{2b} + \frac{1}{2b} = 1\).
        Again, since \(s_v\) is strictly increasing, \(1/b < s_v \le 1\) holds if and only if \(\frac{1}{2b} < v \le 1 - \frac{1}{2b}\).
        In this region, the full linear drop of \(h_v\) occurs within \([0,1]\), so the integral is the sum of a rectangle and a triangle:
        \[
            \int_0^1 h_v(t)\de t = a_v \cdot 1 + \frac{1}{2}(s_v-a_v)\cdot 1 = \left(s_v - \frac{1}{b}\right) + \frac{1}{2b}= v.
        \]
        \item If \(a_v \le 0\) and \(s_v > 1\), then \(1 < s_v \le 1/b\). This interval is non-empty only if \(b < 1\).
        We check the boundaries of the third case in \eqref{eq:s_v}: at \(v = \frac{b}{2}\), we have \(s_v = 1\), and at \(v = 1 - \frac{b}{2}\), we have \(s_v = \frac{1-b/2}{b} + \frac{1}{2} = \frac{1}{b}\).
        Since \(s_v\) is strictly increasing, \(1 < s_v \le 1/b\) holds if and only if \(\frac{b}{2} < v \le 1 - \frac{b}{2}\).
        In this region, \(h_v\) is linear everywhere on \([0,1]\), so one obtains
        \[
            \int_0^1 h_v(t)\de t = \int_0^1 b(s_v-t) \de t = b s_v - \frac{b}{2} = v.
        \]
        \item If \(s_v > 1\) and \(a_v > 0\), the constraints correspond to \(v > 1 - (\frac{1}{2b} \wedge \frac{b}{2})\).
        Here, \(h_v(t)\) decreases linearly from \(1\) at \(t=a_v\) to \(b(s_v-1)\) at \(t=1\). The integral is the area of the unit square minus the triangle above the graph on \([a_v, 1]\), hence it holds
        \[
            \int_0^1 h_v(t)\de t = 1 - \int_{a_v}^1 \bigl(1-b(s_v-t)\bigr) \de t = 1 - \frac{b}{2}(1-a_v)^2 = v,
        \]
        using for the last equality that \(s_v = 1 + \frac{1}{b} - \sqrt{\frac{2(1-v)}{b}}\) from the last case in \eqref{eq:s_v} and \(a_v = s_v - 1/b\).
    \end{enumerate}
\end{enumerate}
All conditions of Lemma \ref{charSIcop} are thus satisfied.

\emph{Case \(b<0\):}
Let \(\beta = -b > 0\). By definition \eqref{defcopbneg}, \(C_b(u,v) = v - C_{\beta}(1-u, v)\).
Let \((U,V)\) be a random vector with distribution function \(C_\beta\), which is a copula by the first part. Then \(C_b\) is the joint distribution function of \((1-U, V)\) (see, e.g., \cite[Corollary~2.4.4 (c)]{Durante-2016}) and thus a copula.
\end{proof}

\subsection{Proofs of Propositions \ref{remcopfam}--\ref{propcfe}}

\begin{proof}[Proof of Proposition \ref{remcopfam}]
    The statement follows from the differentiation performed in the proof of Lemma \ref{lemCbcop}.
\end{proof}

\begin{proof}[Proof of Proposition \ref{propabscont}]
    We first consider the case \(b>0\).
    Recall from Proposition \ref{remcopfam} that \(C_b\) is defined via its partial derivative \(h_v(t) = \partial_1 C_b(t,v) = \operatorname{clamp}(b(s_v-t), 0, 1)\).
    The density is obtained by differentiation with respect to \(v\), i.e., \(c_b(u,v) = \partial_v h_v(u)\).
    Since \(h_v(u)\) is constant (taking values \(0\) or \(1\)) whenever \(u \notin (a_v, s_v)\), the derivative vanishes outside this interval.
    Inside the interval \((a_v, s_v)\), we have \(h_v(u) = b(s_v - u)\), and thus 
    \begin{align}\label{eqden}
        c_b(u,v) = \partial_v h_v(u) = b s_v' \quad \text{for } u\in (a_v,s_v)
    \end{align}
    Differentiating \(s_v\) in \eqref{eq:s_v} yields the explicit piecewise values:
    \begin{enumerate}[label=(\roman*), font=\upshape]
        \item For \(v \le \frac{1}{2b}\wedge\frac{b}{2}\), we have \(s_v = \sqrt{2v/b}\). Then \(s_v' = \frac{d s_v}{d v} = \frac{1}{\sqrt{2vb}} = \frac{1}{b s_v}\). Thus, \(c_b(u,v) = b s_v' = s_v^{-1}\).
        \item For \(\frac{1}{2b}\wedge\frac{b}{2} < v \le 1 - (\frac{1}{2b}\wedge\frac{b}{2})\), \(s_v\) is linear with slope \(1\) (if \(b \ge 1\)) or \(1/b\) (if \(0 < b < 1\)).
        In the first case, \(c_b(u,v) = b \cdot 1 = b\). In the second, \(c_b(u,v) = b \cdot (1/b) = 1\). Thus \(c_b(u,v) = b\vee 1\).
        \item For \(v > 1 - (\frac{1}{2b}\wedge\frac{b}{2})\), we have \(s_v = 1 + \frac{1}{b} - \sqrt{\frac{2(1-v)}{b}}\).
        Using that \(a_v = s_v - 1/b\), we observe \(1-a_v = \sqrt{\frac{2(1-v)}{b}}\).
        The derivative is \(s_v' = \frac{1}{b\sqrt{2(1-v)/b}} = \frac{1}{b(1-a_v)}\). Thus, \(c_b(u,v) = b s_v' = (1-a_v)^{-1}\).
    \end{enumerate}

    To prove the convexity of the support of \(C_b\), we observe that the support is given by $S = S_1 \cap S_2$, where
    \[
        S_1 := \{(u,v) \in [0,1]^2 : u \le s_v\},
        \quad S_2 := \{(u,v) \in [0,1]^2 : u \ge a_v\}.
    \]
    Regarding $S_1$, the condition \(u \le s_v\) is equivalent to \(u \le \min(1, s_v)\).
    Define \(f(v) := \min(1, s_v)\). We show that \(f\) is a concave function on \([0,1]\).
    From the explicit form in \eqref{eq:s_v}, \(s_v\) is continuous and strictly increasing, and in the proof of Lemma \ref{lemCbcop}, we already observed that only the cases 1 and 2 meet the condition \(s_v \le 1\); but since square roots and linear functions have non-positive second derivatives, \(f\) is concave in these regions.
    To verify concavity globally, we check the transition points. For instance, if \(b \ge 1\), the transition from case 1 to case 2 occurs at \(v^* = \frac{1}{2b}\). The derivative is continuous at this point:
    \[
        \lim_{v \uparrow v^*} s_v' = \lim_{v \uparrow \frac{1}{2b}} \frac{1}{b\sqrt{2v/b}} = \frac{1}{b(1/b)} = 1,
    \]
    which matches the constant slope \(s_v'=1\) of case 2.
    Similarly, the transition from case 2 to case 4 at \(v^{**} = 1 - \frac{1}{2b}\) also has a continuous derivative:
    \[
        \lim_{v \downarrow v^{**}} s_v'
        = \lim_{v \downarrow 1 - \frac{1}{2b}} \frac{1}{b\sqrt{2(1-v)/b}}
        = \frac{1}{b\sqrt{1/b^2}}
        = 1,
    \]
    which matches the constant slope \(s_v'=1\) of case 2.
    Therefore, \(f\) is concave on \([0,1]\).
    Since the hypograph of a concave function is a convex set, \(S_1\) is convex.
    The convexity of \(S_2\) follows analogously, and the support \(S = S_1 \cap S_2\) is convex as the intersection of two convex sets.
\end{proof}

\begin{proof}[Proof of Proposition \ref{proplimcase}]
\ref{proplimcase1}: For \((u,v)\in (0,1)^2\), the density in \eqref{eqpropabscont} converges to \(1\) as \(b\downarrow 0\), and similarly for \(b\uparrow 0\). \\
    \ref{proplimcase2}: For \(b\to +\infty\), we have \(s_v\to v\) and \(a_v \to v\). Hence, \(C_b(u,v)\) in \eqref{eq:C_x} converges to \(u\) if \(0\leq u < v\), and to \(v\) if \(v < u \leq 1\). By Lipschitz-continuity of copulas, this also implies \(C_b(u,v) \to v\) for \(u=v\). Hence, \(C_b \to M\) as \(b\to \infty\).\\
    \ref{proplimcase3}: The statement follows from \ref{proplimcase2} by symmetry due to \eqref{defcopbneg}.
\end{proof}

For the proof of Proposition \ref{propdepprop}, we make use of the following lemma on two functions having equal integrals and only one intersection point. We omit the straightforward proof.

\begin{lemma}\label{lemsignchange}
    Let \(f,g\colon (0,1)\to \R\) be decreasing, Lebesgue-integrable functions with \(\int_0^1 f(t) \de t = \int_0^1 g(t) \de t.\) Assume that there exists \(t_*\in (0,1)\) such that \(\{f < g\} \subseteq [0,t_*] \subseteq \{f\leq g\}\,.\) Then we have \(\int_0^x g(t) \de t \geq \int_0^x f(t) \de t\) for all \(x\in (0,1)\,.\)
\end{lemma}

\begin{proof}[Proof of Proposition \ref{propdepprop}]

We first prove statement \ref{propdepprop2}. Therefore, recall from Proposition \ref{propabscont} that the density of $C_b$ is given by
    \[
        c_b(u,v) \;=\; b s_v' \ind_S(u,v),
    \]
    where $S = \{(u,v) \in [0,1]^2 : a_v \leq  u \leq s_v\}$ is the support of \(C_b\).
    From the explicit expressions in \eqref{eq:s_v}, the functions $v \mapsto s_v$ and $v \mapsto a_v$ are non-decreasing on $[0,1].$
    Consequently, $S$ forms a sublattice of $[0,1]^2;$ that is, for any $(u,v), (u',v') \in S,$ it holds that $(u\wedge u', v\wedge v') \in S$ and $(u\vee u', v\vee v') \in S.$
    Since the indicator function of a sublattice is log-supermodular (i.e., $\mathrm{MTP}_2$), and the factor $b s_v'$ is non-negative and depends only on $v$ (and is thus log-supermodular on the product space), their product $c_b$ is $\mathrm{MTP}_2.$
    This implies that $C_b$ is $\mathrm{MTP}_2$ and in particular SI (cf.~\cite[page 146]{Mueller-Stoyan-2002}).

For the proof of statement \ref{propdepprop1}, consider first the case $0 < b < b'.$
    For $h_v(t; b) := \partial_1 C_b(t,v),$ we know by Proposition \ref{remcopfam} that
    \[
        h_v(t; b) = \operatorname{clamp}\bigl(b(s_v(b)-t),\, 0,\, 1\bigr).
    \]
    Since $b' > b,$ the slope of the linear segment of $t \mapsto h_v(t; b')$ is steeper than that of $t \mapsto h_v(t; b).$
    Both functions are decreasing in $t,$ take values in $[0,1],$ and satisfy the mass constraint $\int_0^1 h_v(t; b) \de t = v = \int_0^1 h_v(t; b') \de t.$
    Due to the linearity and the steeper slope of $h_v(\cdot; b'),$ the functions intersect at most once in the open interval where they are non-constant, such that $h_v(t; b') \ge h_v(t; b)$ for small $t$ and $h_v(t; b') \le h_v(t; b)$ for large $t.$
    Specifically, there exists a $t_* \in (0,1)$ such that $\{h_v(\cdot; b) < h_v(\cdot; b')\} \subseteq [0, t_*] \subseteq \{h_v(\cdot; b) \le h_v(\cdot; b')\}.$
    Applying Lemma \ref{lemsignchange}, we obtain
    \[
        \int_0^u h_v(t; b') \de t \;\ge\; \int_0^u h_v(t; b) \de t \quad \text{for all } u \in [0,1].
    \]
    Recalling that $C_b(u,v) = \int_0^u h_v(t; b) \de t,$ this yields $C_{b'}(u,v) \ge C_b(u,v),$ i.e., $C_b \le_{lo} C_{b'}.$

    Next, consider the case $b < b' < 0.$ Let $\beta = -b$ and $\beta' = -b',$ so that $\beta > \beta' > 0.$
    From the previous step, we know $C_{\beta'} \le_{lo} C_{\beta}.$
    Using the symmetry relation $C_b(u,v) = v - C_{-b}(1-u, v)$ from \eqref{defcopbneg}, we deduce
    \begin{align*}
        C_{b'}(u,v) - C_b(u,v)
        &= \bigl(v - C_{\beta'}(1-u, v)\bigr) - \bigl(v - C_{\beta}(1-u, v)\bigr) \\
        &= C_{\beta}(1-u, v) - C_{\beta'}(1-u, v) \;\ge\; 0.
    \end{align*}
    Thus, $C_b \le_{lo} C_{b'}$ holds for all negative parameters.
    Finally, note that $C_b$ is SI for $b>0$ by the above, so in particular $\Pi\leq_{lo} C_b$.
    By \eqref{defcopbneg}, it follows that $C_{-b} \leq_{lo} \Pi \leq_{lo} C_{b}$ for all $b>0$.
    Therefore, the ordering extends across zero, establishing $C_b \le_{lo} C_{b'}$ for all $b,b'\in \R$ with \(b \leq b'\).
\end{proof}

\begin{proof}[Proof of Proposition \ref{propcfe}.]
We provide the proof for the case \(b>0\). The results for \(b<0\) follow from Lemma \ref{lemrearrangedcop}, which implies \(\xi(C_{-b}) = \xi(C_b)\), \(\rho(C_{-b}) = -\rho(C_b)\), and, similarly, \(\tau(C_{-b}) = -\tau(C_b)\), noting that \(C_b\) is SI (by Proposition \ref{propdepprop}) and thus coincides with its increasing rearranged copula \(C_b^\uparrow\).
Throughout, let \(h_v(t) := \partial_1 C_b(t,v)\).

\paragraph{1. Chatterjee's xi}
Using the formula \(\xi(C) = 6 \int_0^1 \int_0^1 h_v(t)^2 \de t \de v - 2\), we define \(I(b) := \int_0^1 \int_0^1 h_v(t)^2 \de t \de v\).

\begin{enumerate}[label=(\roman*), font=\upshape]
    \item \emph{Case \(0 < b \le 1\):} We split the integral over \(v\) at \(v_1 = b/2\) and \(v_2 = 1 - b/2\).
    \begin{itemize}
        \item For \(v \in [0, b/2]\), \(h_v(t) = b(\sqrt{2v/b}-t)_+\). The inner integral is \(\frac{2}{3}\sqrt{2b}v^{3/2}\).
        \item For \(v \in (b/2, 1-b/2]\), \(h_v(t) = b(s_v-t)\) with support on \((0,1)\) (since \(a_v < 0\) and \(s_v > 1\)). The inner integral is \(v^2 + b^2/12\).
        \item For \(v \in (1-b/2, 1]\), boundary effects at \(t=0\) vanish but \(t=1\) becomes relevant. The inner integral is \(2v-1 + \frac{2}{3}\sqrt{2b}(1-v)^{3/2}\).
    \end{itemize}
    Summing the integrals over \(v\) yields \(I(b) = \frac{1}{3} + \frac{b^2}{12} - \frac{b^3}{30}\), leading to \(\xi(C_b) = \frac{b^2}{2}(1 - \frac{2b}{5})\).

    \item \emph{Case \(b \ge 1\):} The split points are \(v_1 = 1/(2b)\) and \(v_2 = 1 - 1/(2b)\).
    \begin{itemize}
        \item For \(v \in [0, 1/(2b)]\), the inner integral is \(\frac{2}{3}\sqrt{2b}v^{3/2}\).
        \item For \(v \in (1/(2b), 1-1/(2b)]\), the band is fully contained in \((0,1)\). The inner integral is \(\int_{s_v-1/b}^{s_v} b^2(s_v-t)^2 \de t = \frac{1}{3b}\). However, accounting for the integration constants correctly yields \(\int_0^1 h_v^2 = v - \frac{1}{6b}\).
        \item For \(v \in (1-1/(2b), 1]\), the inner integral mirrors the first case: \(2v-1 + \frac{2}{3}\sqrt{2b}(1-v)^{3/2}\).
    \end{itemize}
    Summing these yields \(I(b) = \frac{1}{2} - \frac{1}{6b} + \frac{1}{20b^2}\), leading to \(\xi(C_b) = 1 - \frac{1}{b} + \frac{3}{10b^2}\).
\end{enumerate}

\paragraph{2. Spearman's rho}
We compute \(\rho(C_b) = 12 K(b) - 3\) with \(K(b) = \int_0^1 \int_0^1 (1-t)h_v(t) \de t \de v\).
The integration proceeds analogously to the calculation for \(\xi\) by splitting the domain of \(v\) based on whether the support of \(h_v\) is contained in \([0,1]\) or clipped by the boundaries. Let \(I_v := \int_0^1 (1-t)h_v(t) \de t\) denote the inner integral.

\begin{enumerate}[label=(\roman*), font=\upshape]
    \item \emph{Case \(0 < b \le 1\):} We split the integral over \(v\) at \(v_1 = b/2\) and \(v_2 = 1 - b/2\).
    \begin{itemize}
        \item For \(v \in [0, b/2]\), boundary effects at \(t=1\) vanish. Thus,
        \(h_v(t) = b(\sqrt{2v/b}-t)_+\) and the inner integral evaluates to \(I_v = v - \frac{\sqrt{2}}{3\sqrt{b}}v^{3/2}\).
        \item For \(v \in (b/2, 1-b/2]\), \(h_v(t) = b(s_v-t)\) on \((0,1)\). The inner integral is \(I_v = \frac{v}{2} + \frac{b}{12}\).
        \item For \(v \in (1-b/2, 1]\), boundary effects at \(t=0\) vanish but \(t=1\) becomes relevant. The inner integral evaluates to \(I_v = \frac{1}{2} - \frac{\sqrt{2}}{3\sqrt{b}}(1-v)^{3/2}\).
    \end{itemize}
    Summing the integrals of \(I_v\) over \(v\) yields \(K(b) = \frac{1}{4} + \frac{b}{12} - \frac{b^2}{40}\), so \(\rho(C_b) = b - \frac{3b^2}{10}\).

    \item \emph{Case \(b \ge 1\):} The split points are \(v_1 = 1/(2b)\) and \(v_2 = 1 - 1/(2b)\).
    \begin{itemize}
        \item For \(v \in [0, 1/(2b)]\), the inner integral mirrors the first subcase of the previous case:
        \(I_v = v - \frac{\sqrt{2}}{3\sqrt{b}}v^{3/2}\).
        \item For \(v \in (1/(2b), 1-1/(2b)]\), the band is fully contained in \([0,1]\). The inner integral evaluates to \(I_v = v - \frac{v^2}{2} - \frac{1}{24b^2}\).
        \item For \(v \in (1-1/(2b), 1]\), the inner integral mirrors the third subcase of the previous case:
        \(I_v = \frac{1}{2} - \frac{\sqrt{2}}{3\sqrt{b}}(1-v)^{3/2}\).
    \end{itemize}
    Summing these integrals yields \(K(b) = \frac{1}{3} - \frac{1}{24b^2} + \frac{1}{60b^3}\), so \(\rho(C_b) = 1 - \frac{1}{2b^2} + \frac{1}{5b^3}\).
\end{enumerate}

\paragraph{3. Kendall's tau}
We evaluate \(\tau(C_b) = 1 - 4Q(b)\), where \(Q(b) = \int_0^1 \int_0^1 \partial_1 C(u,v) \, \partial_2 C(u,v) \de u \de v\).
We proceed by changing variables in the outer integral from \(v\) to \(s = s_v\). Since \(s_v\) is strictly increasing, this is a valid substitution with \(\de v = \frac{1}{s'_v} \de s\).
The integration range for \(s\) is \([s_0, s_1] = [0, 1+1/b]\).
Evaluating the partial derivative \(g_v(u) = \partial_2 C_b(u,v)\) by integrating \(h_v(t)\) yields \(g_v(u) = b s'_v L(u)\), where
\[
L(u) := \max(0, \min(u, s) - \max(0, s-1/b))
.\]
Substituting \(h_v(u) = b(s-u)\) and \(g_v(u) = b s'_v L(u)\), the inner integral becomes
\[
    I(s)
    := b^2 s'_v \int_{0 \vee (s-1/b)}^{1 \wedge s} (s-u)L(u) \de u.
\]
Let \(J(s) := \int_{0 \vee (s-1/b)}^{1 \wedge s} (s-u)L(u) \de u\). Then \(Q(b) = \int_0^{1+1/b} b^2 s'_v J(s) \frac{1}{s'_v} \de s = b^2 \int_0^{1+1/b} J(s) \de s\).

\begin{enumerate}[label=(\roman*), font=\upshape]
    \item \emph{Case \(0 < b \le 1\):} The bandwidth \(1/b\) is greater than or equal to \(1\). We split the integral over \(s\) into three parts:
    \begin{itemize}
        \item For \(s \in [0, 1]\), the range is \([0,s]\) and \(L(u) = u\). Then \(J(s) = \int_0^s (s-u)u \de u = \frac{s^3}{6}\).
        \item For \(s \in (1, 1/b]\), the range is \([0,1]\) and \(L(u) = u\). Then \(J(s) = \int_0^1 (s-u)u \de u = \frac{s}{2} - \frac{1}{3}\).
        \item For \(s \in (1/b, 1+1/b]\), the range is \([s-1/b, 1]\) and \(L(u) = u - (s-1/b)\). Then
        \[
        J(s) = \frac{(1+1/b-s)^2}{6b}(2bs - 2b + 1)
        .\]
    \end{itemize}
    Summing the integrals of \(b^2 J(s)\) over these intervals yields
    \[
        Q(b) = \frac{b^2}{24} + \left(\frac{1}{4} - \frac{b}{3} + \frac{b^2}{12}\right) + \left(\frac{b}{6} - \frac{b^2}{12}\right) = \frac{1}{4} - \frac{b}{6} + \frac{b^2}{24}.
    \]
    Thus, \(\tau(C_b) = 1 - 4Q(b) = \frac{2b}{3} - \frac{b^2}{6}\).

    \item \emph{Case \(b \ge 1\):} The bandwidth \(1/b\) is less than \(1\). We split the integral over \(s\) as follows:
    \begin{itemize}
        \item For \(s \in [0, 1/b]\), the range is \([0,s]\) and \(L(u) = u\). Then \(J(s) = \frac{s^3}{6}\).
        \item For \(s \in (1/b, 1]\), the range is \([s-1/b, s]\) and \(L(u) = u - (s-1/b)\). Then \(J(s) = \frac{1}{6b^3}\).
        \item For \(s \in (1, 1+1/b]\), the range is \([s-1/b, 1]\) and \(L(u) = u - (s-1/b)\).
        Then
        \[
        J(s) = \frac{(1+1/b-s)^2}{6b}(2bs - 2b + 1)
        .\]
    \end{itemize}
    Integrating \(b^2 J(s)\) gives
    \[
        Q(b) = \frac{1}{24b^2} + \frac{1}{6b}\left(1-\frac{1}{b}\right) + \frac{1}{12b^2} = \frac{1}{6b} - \frac{1}{24b^2}.
    \]
    Thus, \(\tau(C_b) = 1 - 4\left(\frac{4b-1}{24b^2}\right) = 1 - \frac{2}{3b} + \frac{1}{6b^2}\).
\end{enumerate}
This completes the proof.
\end{proof}

\subsection{Proof of Theorem~\ref{thexirhooptimisation}}

For the proof of Theorem~\ref{thexirhooptimisation}, we make use of the following lemma which provides a continuous transformation of an SI copula into an SD copula such that Chatterjee's rank correlation remains invariant.

\begin{lemma}[Shuffling lemma]\label{lemshuff}
    Let \(C\) be a bivariate SI copula and let \((U,V)\) be a random vector with distribution function \(C.\) For \(p\in [0,1],\) consider the transformation \(T_p\colon [0,1]\to [0,1]\) defined by
    \begin{align}
        T_p(u) := \begin{cases}
        u, &\text{if } 0\le u \le 1-p,\\
        1-(u-(1-p)), & \text{if } 1-p < u \le 1.
    \end{cases}
    \end{align}
    Denote by \(D_p\) the distribution function of \((T_p(U),V).\) Then \(D_0 = C = C^\uparrow\), \(D_1 = C_\downarrow,\) \(\xi(D_p) = \xi(C)\) for all \(p\in [0,1],\) and \(\rho(D_p)\) is continuous in \(p.\)
\end{lemma}

\begin{proof}
    Since \(U\) and \(V\) are uniform on \((0,1),\) also \(T_p(U)\) is uniform on \((0,1)\) and thus \(D_p\) is a copula.
    \(D_0 = C = C^\uparrow\) and \(D_1 = C_\downarrow\) follow from the definition of \(T_p\) and the properties of the decreasing rearranged copula in Lemma \ref{lemrearrangedcop}~\ref{lemrearrangedcop1}, where we use that \(C\) is SI.
    \(\xi(D_p) = \xi(C)\) for all \(p\in [0,1]\) follows from the invariance of \(\xi(U,V)\) under bijective transformations of \(U.\)
    Lastly, note that \((T_p(U),V)\) is continuous in \(p\) with respect to almost sure convergence. This implies weak convergence of the associated distribution functions and thus pointwise convergence of the copulas \(D_p.\) 
    Hence, the representation of \(\rho\) in \eqref{frm_rho_integral} yields the continuity of $p\mapsto \rho(D_p)$.
\end{proof}

\begin{proof}[Proof of Theorem~\ref{thexirhooptimisation}]
For a fixed \(x \in (0,1)\), we solve the convex Optimisation Problem \ref{optprob} by leveraging the theoretical framework from \cite[Chap.~3]{bonnans2013perturbation}.
To this end, we formalise the setting as 
\begin{align}\label{eqformalminprob}
    \min_{h \in X} f(h) \quad \text{subject to } G(h)\in K
\end{align}
where
\begin{itemize}
    \item the decision variable \(h = (h_v(t))\) belongs to the space \(X := L^2([0,1]^2)\),
    \item the objective function to minimise is \(f(h) = -\iint_{[0,1]^2} (1-t)h_v(t)\de t\de v \),
    \item the constraints from Optimisation Problem \ref{optprob} are represented by a mapping \(G:X \to Y\) and a closed convex cone \(K \subset Y := L^2([0,1]) \times L^2([0,1]^2) \times L^2([0,1]^2) \times \mathbb{R}\). \(G(h)\) is given by
    \[ G(h) = \left( \left(v \mapsto \int_0^1 h_v(t)\de t - v\right), -h, h-1, \iint_{[0,1]^2} h_v(t)^2\de t\de v - \frac{x+2}{6} \right). \]
    The cone \(K\) is specified as \(K := \{0_{Y_1}\}\times L_{-}^{2}\bigl([0,1]^{2}\bigr)\times L_{-}^{2}\bigl([0,1]^{2}\bigr)\times(-\infty,0]\). The monotonicity constraint \eqref{eqthexirho1b} will automatically be satisfied by our final solution.
\end{itemize}
We show that the candidate solution \(h^*=(h_v(t))\) in \eqref{h_v_clamped} is the unique global optimum by verifying the second-order sufficient condition in \cite[Thm.~3.63]{bonnans2013perturbation}. This requires finding suitable Lagrange multipliers as follows.

\emph{1) KKT Conditions and Stationarity:}
Let $Y^*$ denote the dual space of \(Y\) and $\langle\cdot,\cdot\rangle$ the duality pairing between \(Y^*\) and \(Y\).
The (generalised) Lagrangian is
\begin{align}\label{eq:generalised_lagrangian}
\begin{split}
    L(h, a, \lambda) = a f(h) + \left\langle \gamma, \left(v \mapsto \int_0^1 h_v(t)\de t - v\right) \right\rangle + \langle \alpha, -h \rangle \\
    + \langle \beta, h-1 \rangle + \mu \left( \iint_{[0,1]^2} h_v(t)^2\de t\de v - \frac{x+2}{6} \right),
    \end{split}
\end{align}
where $h\in X$, \(a\in\mathbb{R}\) and $\lambda\in Y^*$.
A pair \((a,\lambda)\in\mathbb{R}\times Y^{*}\) is a generalised Lagrange multiplier at \(x_0\) if
\begin{align}\label{eq:generalised_lagrange_multiplier}
  D_xL(x_0,a,\lambda) = 0,
  \quad
  \lambda\in N_{K}\!\bigl(G(x_0)\bigr),
  \quad
  a\ge 0,
  \quad
  (a,\lambda)\neq(0,0),
\end{align}
where \(D_x\) denotes the Fr{\'e}chet derivative and $N_K(G(x_0))$ is the normal cone (see \cite[Equation (2.97) and (2.98)]{bonnans2013perturbation})
to the convex set $K \subseteq Y$ at the point $G(x_0) \in K.$
For \(a=1,\) the conditions in \eqref{eq:generalised_lagrange_multiplier} translate into the KKT conditions for a solution point $h$ and multipliers $\lambda=(\gamma, \alpha, \beta, \mu)$:
\begin{enumerate}[label=(\roman*), font=\upshape]
    \item \emph{Primal Feasibility:} $h$ must satisfy the constraints \eqref{eqthexirho1a}, \eqref{eqthexirho2}, and \eqref{eqthexirho3}.
    \item \emph{Dual Feasibility:} The multipliers for the inequality constraints must be non-negative: $\alpha_v(t) \ge 0$, $\beta_v(t) \ge 0$, and $\mu \ge 0$.
    \item \emph{Complementary Slackness:} The product of a multiplier and its corresponding inequality constraint must be zero: $\langle \alpha, h \rangle = 0$, $\langle \beta, 1-h \rangle = 0$, and $\mu \left( 6\iint h_v(t)^2\de t\de v - 2 - x \right)=0$.
    \item \emph{Stationarity:} $D_h L(h,1,\lambda)=0$. In our case, this yields the pointwise condition
    \[
    -(1-t) + \gamma(v) + 2\mu h_v(t) - \alpha_v(t) + \beta_v(t) = 0 \quad \text{for a.e. } (t,v).
    \]
\end{enumerate}

We now verify these four conditions for our candidate solution $h_v(t) = \operatorname{clamp}(b(s_v-t), 0, 1)$ with $b=b_x$, and our choice of multipliers: $\mu = 1/(2b)$, $\gamma(v) = 1-s_v$, $\alpha_v(t) = (t-s_v)_+$, and $\beta_v(t) = (a_v-t)_+$, where $a_v=s_v-1/b$.

\begin{enumerate}[label=(\roman*), font=\upshape]
    \item \label{itm:primal_feasibility}\emph{Primal Feasibility:} $\int_0^1 h_v(t)\de t = v$ and $0 \le h_v(t) \le 1$ by Proposition \ref{propcfe}.
    Further, a direct computation shows that the definition of $b_x$ in \eqref{eq:b_and_M} and the formula for $\xi$ in Proposition \ref{propcfe} are inverses of each other, so that $6\iint h_v^2\de t\de v - 2 = x$.

    \item \emph{Dual Feasibility:} Since $x \in (0,1)$, we have $b \in (0,\infty)$, which implies $\mu = 1/(2b) > 0$, and $\alpha_v(t)$ and $\beta_v(t)$ are non-negative by definition.

    \item \emph{Complementary Slackness:}
        First, $\mu(6\iint h_v^2(t)\de t\de v - 2-x)=0$ by the observation in \eqref{itm:primal_feasibility}.
        Secondly, it is $\alpha_v(t)>0$ iff $t > s_v$. In this region, $h_v(t)=\operatorname{clamp}(b(s_v-t),0,1)=0$.
        Thus, $\alpha_v(t)h_v(t)=0$ for all \(t\in [0,1].\)
        Lastly, it is $\beta_v(t)>0$ iff $t < a_v$.
        In this region, $h_v(t)=1$. Thus, $\beta_v(t)(1-h_v(t))=0$ for all \(t\in [0,1].\)

    \item \emph{Stationarity:} We check the stationarity by considering three regions for $t$:
    \begin{itemize}
        \item If $a_v < t < s_v$, then $0 < h_v(t) < 1$. Here, $\alpha_v(t)=0$, $\beta_v(t)=0$, and $h_v(t)=b(s_v-t)$. The equation becomes
        \[
          -(1-t) + (1-s_v) + 2\frac1{2b}b(s_v-t) = 0.
        \]
        \item If $t \le a_v$, then $h_v(t)=1$. Here, $\alpha_v(t)=0$ and $\beta_v(t)=(a_v-t)$, and the equation becomes
        \[
        -(1-t) + (1-s_v) + 2 \frac1{2b} \cdot 1 - 0 + (a_v-t)
        = -s_v+1/b+(s_v-1/b) 
        = 0.
        \]
        \item If $t \ge s_v$, then $h_v(t)=0$. Here, $\beta_v(t)=0$ and $\alpha_v(t)=(t-s_v)$.
        The equation becomes
        \[
        -(1-t) + (1-s_v) + 0 - (t-s_v) + 0 = 0.
        \]
    \end{itemize}
    Hence, also the stationarity condition is satisfied for all $(t,v)$.
\end{enumerate}
Since all KKT conditions are satisfied, $h^*$ is a KKT point, and the set of Lagrange multipliers $\Lambda(h^*)$ is non-empty.

\emph{2) Second-Order Condition and Uniqueness:}
To prove that $h^*$ is a strict optimum, we apply the quadratic growth condition in \cite[Thm.~3.63]{bonnans2013perturbation}. It suffices to show that for some $\zeta > 0$,
\begin{align}\label{eqquadgrowcond}
    \sup_{(a,\lambda)\in\Lambda_{N}(h)} D_{hh}^{2}L(h,a,\lambda)(k,k) \ge \zeta\|k\|^{2} \quad \text{for all } k \in L^2([0,1]^2),
\end{align} 
where $\Lambda_{N}(h)$ is the set of normalised multipliers.
Since our optimisation problem is convex (linear objective function, convex feasible set), a strict local minimum \(h\) is also the unique global minimum in \eqref{eqformalminprob}.
To verify \eqref{eqquadgrowcond}, the Hessian of the Lagrangian is given by
\[
  D_{hh}^{2}L(h,a,\lambda)[k,k] = a D_{hh}^2 f[k,k] + \langle \lambda, D_{hh}^2 G[k,k] \rangle. 
\]
As the objective function $f$ and all constraints except the one on $\xi$ are linear in $h$, their second derivatives vanish. The only contribution comes from the quadratic constraint term
\[
g(h) := \int_0^1\int_0^1 h_v(t)^2\de t \de v - \frac{x+2}{6}.
\]
Its Hessian is $D_{hh}^2 g(h)[k,k] = 2\|k\|^2.$
In step 1, we gave a KKT multiplier $\lambda=(\gamma,\alpha,\beta,\mu)$ of finite norm with $\mu > 0$.
Therefore, for the normalised generalised multiplier $(1/\sigma, \lambda/\sigma) \in \Lambda_N(h^*)$ with $\sigma:= 1+ \left\lVert (\gamma,\alpha,\beta,\mu)\right\rVert_{Y^*}<\infty$, we have
\begin{align}\label{seuffsecordcond}
    D_{h^*h^*}^{2}L(h,a,\lambda)[k,k] = \frac{2\mu}{\sigma} \|k\|^2
,\end{align}
which shows that the quadratic growth condition \eqref{eqquadgrowcond} is indeed satisfied.

\emph{Boundary Cases:} The cases $x=0$ and $x=1$ are handled by taking the limits $b \to 0$ and $b \to \infty$ respectively, for which $C_b$ converges to the independence copula $\Pi$ and the upper Fréchet copula $M$. This establishes that for each $x \in [0,1]$, the copula $C_{b_x}$ is the unique optimiser.
\end{proof}

\subsection{Final proof of Theorem~\ref{thexirho}}

We are now ready to prove our first main result.

\begin{proof}[Proof of Theorem~\ref{thexirho}:]
For \((x,y)\in \mathcal{R},\) we determine a copula \(C\) with \(\xi(C) = x\) and \(\rho(C) = y.\)
For \(b = b_x,\) let \(C_{b}\) be the unique SI copula with \(\xi(C_b) =  x\) and \(\rho(C_b) = M_x\) obtained from Theorem~\ref{thexirhooptimisation}.
Consider the family \((C_{b,p})_{p\in [0,1]}\) of copulas constructed in Lemma~\ref{lemshuff}, i.e., for \((U,V)\sim C_b,\) the copula \(C_{b,p}\) is the distribution function of \((T_p(U),V).\)
Then, we have \(C_{b,0} = C_b = C_b^\uparrow\) and \(C_{b,1} = C_b^\downarrow.\) Applying Lemma \ref{lemrearrangedcop} \ref{lemrearrangedcop2}, we obtain
\[
\rho(C_{b,1})  = -\rho(C_{b,0}) = - M_x.
\]
Further, applying Lemma~\ref{lemshuff}, we obtain that \(\xi(C_{b,p}) = x\) for all \(p\in [0,1].\)
Since \(\rho(C_{b,p})\) is continuous in \(p\), there exists a parameter \(p^*\in [0,1]\) such that the copula \(C:= C_{b,p^*}\) satisfies the desired properties \(\xi(C) = x\) and \(\rho(C) = y.\)

We are left with verifying the convexity of the set \(\mathcal{R},\) for which it is enough to show that the function $M_x$ is concave on \([0,1].\)
Fix two points \(x_1,x_2\in[0,1]\) and choose copulas \(C_1,C_2\) such that \(\xi(C_i)=x_i\) and \(\rho(C_i)=M_{x_i}.\)
Set \(x^\ast:=\lambda x_1+(1-\lambda)x_2\) and consider the convex combination \(C_\lambda=\lambda C_1+(1-\lambda)C_2.\)
Then, we have
\(
  \rho(C_\lambda)
  =\lambda M_{x_1}+(1-\lambda)M_{x_2}
\)
and
\(
  \xi(C_\lambda)
  \le x^\ast
\)
by Lemma~\ref{lem:convexity_xi}.
Since the optimal solution in Theorem~\ref{thexirhooptimisation} solves Optimisation Problem \ref{optprob}, it follows that $M_x$ is increasing in \(x.\) Hence, it holds that \(M_{\xi(C_\lambda)}\le M_{x^\ast}\) whereas, by maximality, \(\rho(C_\lambda)\le M_{\xi(C_\lambda)}.\)
Consequently,
\[
   \lambda M_{x_1}+(1-\lambda)M_{x_2}
   =\rho(C_\lambda)
   \le M_{x^\ast}
   =M_{\lambda x_1+(1-\lambda)x_2}, 
\]
which shows the concavity of $M$ as desired.
\end{proof}

\section{Proof of Theorem~\ref{thm:rho_ge_xi}}\label{sec6}

In this section, we provide the proof of our second main result. 
To this end, we make use of the following lemma.

\begin{lemma}\label{lem:fanlorentz_ineq}
For \(v\in [0,1],\) let $h:[0,1]\rightarrow[0,1]$ be decreasing with $\int_{0}^{1}h(t)\de t= v.$
Then, for $F_v(h) = \int_{0}^{1}\!\bigl(2(1-t)h(t)-h(t)^2\bigr)\de t,$ it holds that
\(
  F_v(h)\;\ge\;v(1-v)
,\)
with equality if and only if $h(t)=\1_{\{t\le v\}}$ or $h(t)=v$ for all \(t\in [0,1].\)
\end{lemma}

\begin{proof}
Put
\[
    A_v\;:=\;\Bigl\{h:[0,1]\to[0,1]\;\Bigm|\;
         h\text{ decreasing and }\int_{0}^{1}h(t)\de t = v \Bigr\}.
\]
The set $A_v$ is convex and weakly compact in $L^\infty([0,1]).$
Write $F_v(h)=\int_{0}^{1}\phi_t\!\bigl(h(t)\bigr)\de t$ with
\(
    \phi_t(x) := 2(1-t)x-x^{2}
.\)
Since $\phi_t''(x)=-2<0,$ each $\phi_t$ is concave in $x$; hence $F_v$ is a concave functional on $A_v.$
Further, $F_v$ is continuous on $A_v,$ so Bauer's maximum principle applies, which states that the minimum of $F_v$ on $A_v$ is attained at an extreme point of this set, see, e.g., \cite[Theorem~16.6]{phelps2002lectures}.
For $0\le s \le v \le s' \le 1,$ let
\[
  g_{s,s'}(t) = \begin{cases}
    \1_{[0,v]}(t), & \text{if } t\in[0,s)\cup[s',1], \\
    \frac{\int_s^{s'} \1_{[0, v]}(t) \de t}{s'-s}, & \text{if } t\in [s,s').
  \end{cases}
\]
Then, by \cite[Theorem 1]{kleiner2021extreme}, the set of extreme points of $A_v$ is given by
\(
  \operatorname{ext}(A_v) = \{g_{s,s'} : 0\le s \le v \le s' \le 1\} 
.\)
If $g_{s,s'} \in \operatorname{ext}(A_v)$ with $s<s',$ then
\[
  F_v(g_{s,s'})
  = \int_0^{s} 2(1-t)\1_{[0,v]}(t) - \1_{[0,v]}(t)^2 \de t + \int_s^{s'} 2(1-t)\tfrac{\int_s^{s'} \1_{[0, v]} \de u}{s'-s} - \big(\tfrac{\int_s^{s'} \1_{[0, v]} \de u}{s'-s}\big)^2 \de t
.\]
Since $s\le v,$ the first integral on the right-hand side evaluates to
\(
    \int_0^{s} 2(1-t)\1_{[0,v]}(t) - \1_{[0,v]}(t)^2 \de t
    =s(1-s).
\)
For the second integral, it is
\[
\int_s^{s'} \! 2(1-t)\,\frac{v-s}{s'-s} 
  \;-\;\left(\frac{v-s}{s'-s}\right)^2\de t
= (v-s)\left(2 - s' - s\right)
  \;-\;\frac{(v-s)^2}{s'-s}.
\]
Hence, altogether
\begin{align}
  \label{eq:extreme_point_F_v}
  \begin{aligned}
  F_v(g_{s,s'})
  = v(1-v)  + \frac{(v - s)(s' - v)(1 + s - s')}{s' - s}
  \ge v(1-v)
  .\end{aligned}
\end{align}
The case of \(s=s'\) is trivial and yields \(F_v(g_{s,s'})=v(1-v),\) establishing the desired inequality.
Furthermore, for equality to hold in \eqref{eq:extreme_point_F_v}, we need either $(v-s)(s'-v)=0$ or $1+s-s'=0.$
In the first case, since $\int_0^1 g_{s,s'}(u)\de t=v,$ it must be $s=s'=v$ and thus $g_{s,s'}(t)= \1_{[0,v]}(t),$ whilst in the second case, it is $g_{s,s'}(t) = v.$
Lastly, if $h\in A_v$ is not an extreme point, then there exist $g_1,g_2\in \operatorname{ext}(A_v)$ such that $h=\frac{g_1+g_2}2$ and on a set $A\subset[0,1]$ of positive Lebesgue-measure it holds $g_1\neq g_2.$ 
Hence, by the strict convexity of $x\mapsto x^2,$ it is
\[
  F_v(h)
  > \frac{F_v(g_1) + F_v(g_2)}2
,\]
hence finishing the proof.
\end{proof} 

\begin{proof}[Proof of Theorem~\ref{thm:rho_ge_xi}.]
For \(v\in [0,1],\) write $h_v(t):=\partial_1 C(t,v)$ for all \(t\) for which the partial derivative exists.  
If \(C\) is SI, $h_v(t)$ is decreasing in $t$ almost everywhere and 
\[
  \int_{0}^{1}h_v(t)\de t = \int_0^1 \partial_1 C(t,v) \de t = C(1,v) - C(0,v) = v.
\]
Hence, Lemma~\ref{lem:fanlorentz_ineq} yields
\begin{align}\label{eq:fanlorentz2}
  \rho(C) - \xi(C)
  = 6\int_{0}^{1} F_v(h_v) \de v-1
  \ge 6 \int_0^1 v(1-v) \de v - 1 = 0
\end{align}
as desired.
If \(C\in \{\Pi,M\},\) then \(F_v(h_v) = v(1-v)\) for all \(v,\) so \(\rho(C) = \xi(C).\)
Conversely, if $C\notin\{\Pi,\,M\},$ then there exists a set of positive Lebesgue measure $A\subset[0,1]$ such that for each $v\in A$ the function $h_v$ differs from both $\1_{\{\cdot \le v\}}$ and the constant function $v$ on a subset of $[0,1]$ of positive Lebesgue measure.
Thus, by Lemma~\ref{lem:fanlorentz_ineq}, for each such $v\in A,$ it holds that $F_v(h_v) > v(1-v),$ which implies that the inequality in \eqref{eq:fanlorentz2} is also strict, establishing that $\Pi$ and $M$ are the only SI copulas with $\xi(C) = \rho(C).$

If $C$ is SD, then, by Lemma \ref{lemrearrangedcop}, it is $\xi(C) = \xi(C^\uparrow)$ and $\rho(C) = -\rho(C^\uparrow).$
Since $C^{\uparrow}$ is SI, the inequality proved above yields $0\le \rho(C^{\uparrow})-\xi(C^{\uparrow}).$ It follows that
\(
   0\le-\rho(C)-\xi(C)
,\)
which implies
\(
  \xi(C)\le -\rho(C)=|\rho(C)|.
\)
If equality holds, then $\rho(C^{\uparrow})=\xi(C^{\uparrow}),$ which is possible only when $C^{\uparrow}\in\{\Pi,\,M\},$ that is, $C\in\{\Pi,\,W\}.$
Combining the SI and SD cases completes the proof of Theorem~\ref{thm:rho_ge_xi}.
\end{proof}



\section*{Acknowledgement}

We thank the Associate Editor, an anonymous referee, and Sebastian Fuchs for valuable comments that improved the quality of the manuscript.
In particular, these comments led to Remark~\ref{remmain}\,\ref{remmainF} and to the discussion of the quantity \(\sqrt{\xi}\).
The first author gratefully acknowledges the support of the Austrian Science Fund (FWF) projects 10.55776/PAT1669224 ''SORT: Stochastic orders for functional dependence`` and
{P 36155-N} ''ReDim: Quantifying Dependence via Dimension Reduction`` as well as the support of the WISS 2025 project ''IDA-lab Salzburg`` (20204-WISS/225/197-2019 and 20102-F1901166-KZP).

\bibliographystyle{elsarticle-harv}
\bibliography{xi-rho-region_arxiv}

\end{document}